\documentclass[12pt]{iopart}
\eqnobysec 

\usepackage{multirow}
\usepackage{amsthm} 
\usepackage{iopams}  %AMS fonts required
\usepackage{graphicx}
\def\R{\mathbb{R}}

\renewcommand{\phi}{\varphi}
\newcommand{\eps}{\varepsilon}

\usepackage{xcolor}

\newtheorem{dfn}{Definition}

\newtheorem{theorem}{Theorem}[section]
\newtheorem{proposition}[theorem]{Proposition}

\newtheorem{assumption}[theorem]{Assumption}

\newtheorem{remark}[theorem]{Remark}

\usepackage{marginnote}
\usepackage[colorinlistoftodos]{todonotes}
\definecolor{electricindigo}{rgb}{0.44, 0.0, 1.0}
\definecolor{ProfRed}{rgb}{0.7 0.1 0.1}
\definecolor{RegGreen}{rgb}{0,0.4,0}
\definecolor{ConfGreen}{rgb}{0,1,0}
\definecolor{SimBlue}{rgb}{0,0.6,0.6}

\newcommand{\req}[1]{(\ref{#1})}
\newcommand{\dup}[2]{\langle{#1},{#2}\rangle_{X^*,X}}
\begin{document}
%%%%%%%%%%%%%%%%%%%%
%\title[An efficient method for uncertainty quantification in inverse problems]{A fast and exact method for uncertainty quantification in inverse problems with partial observations and PDE constraints}
%\title[A fast and accurate method for uncertainty quantification in parameter estimation problems]{A fast and accurate method for uncertainty quantification in parameter estimation problems with PDE constraints}
\title[Integration based profile likelihood calculation for PDE constrained problems]{Integration based profile likelihood calculation for PDE constrained parameter estimation problems}

\author{R Boiger$^1$, J Hasenauer$^{2,3}$, S Hro\ss$^{2,3}$ and B Kaltenbacher$^1$}

\address{$^1$ Alpen-Adria-Universit\"at Klagenfurt, Universit\"atstra{\ss}e 65-67, A-9020 Klagenfurt, Austria}
\address{$^2$ Helmholtz Zentrum M\"unchen --  German Research Center for Environmental Health, Institute of Computational Biology, Ingoldst\"adter Landstra{\ss}e 1, D-85764 Neuherberg, Germany}
\address{$^3$ Technische Universit\"at M\"unchen, Department of Mathematics, Boltzmannstra{\ss}e 3, D-85748 Garching, Germany}

\ead{\mailto{Romana.Boiger@aau.at}, \mailto{Jan.Hasenauer@helmholtz-muenchen.de}, \mailto{Sabrina.Hross@helmholtz-muenchen.de}, \mailto{Barbara.Kaltenbacher@aau.at}}
%%%%%%%%%%%%%%%%%%%%
\begin{abstract}
%Inverse problems are omnipresent engineering, life sciences and other research fields. Of special interest in those fields is the estimation of underlying process parameters from given measurement data. A key aspect of parameter estimation is the quantification of uncertainty. Commonly, statistical concepts like confidence intervals are used for this purpose, however, those are often calculated based on approximations, which become corrupted in the presence of partial observations or non-identifiability. We have developed a method to calculate reliable confidence intervals for inverse problems with PDE constraints for applications with partial observations based on profile likelihoods. Profile likelihood calculation is often very time consuming or even infeasible for problems with PDE constraints due to the repeated solution of the optimization problem. We propose a new approach to calculate profile likelihoods based on the Hessian matrix of the objective function which drastically reduces the calculation time.  We evaluated the method for a small biologically motivated model, which showed a considerable speed up compared to existing methods. This enables the application of the approach to the computationally demanding parameter estimation in engineering, medical or biological problems.

Partial differential equation (PDE) models are widely used in engineering and natural sciences to describe spatio-temporal processes. The parameters of the considered processes are often unknown and have to be estimated from experimental data. Due to partial observations and measurement noise, these parameter estimates are subject to uncertainty. This uncertainty can be assessed using profile likelihoods, a reliable but computationally intensive approach. In this paper, we introduce an integration based approach for the profile likelihood calculation for inverse problems with PDE constraints. While existing approaches rely on repeated optimization, the proposed approach exploits a dynamical system evolving along the likelihood profile. We derive the dynamical system for the reduced and the full estimation problem and study its properties. To evaluate the proposed method, we compare it with state-of-the-art algorithms for a simple reaction-diffusion model for a cellular patterning process. We observe a good accuracy of the method as well as a significant speed up as compared to established methods. Integration based profile calculation facilitates rigorous uncertainty analysis for computationally demanding parameter estimation problems with PDE constraints.

\end{abstract}

%\submitto{\IP}
\maketitle

%%%%
%%%%
%%%%
\section{Introduction}
\label{sec:Intro}

Engineering, physics, biology and adjacent fields employ PDE models for the mathematical description of involved processes. These models often contain unknown parameters which have to be inferred from experimental data. The corresponding parameter estimation problems are potentially ill-posed due to limited and noise-corrupted experimental data \cite{Hadamard1902}. Due to the ill-posedness a comprehensive uncertainty analysis is crucial.
In particular, we refer to \cite{EngletalSysbio} for an overview on inverse problems in systems biology with an emphasis on regularization aspects.
Since we here deal with finite dimensional parameter spaces, regularization (see, e.g., \cite{EHNBuch}) is not required. Still we face the difficulty that some parameters might not be uniquely determined from the given noisy measurements and due to the strong nonlinearity of the problem, this indeterminacy might not be detectable by just considering the nullspace of the linearized forward operator. Thus we here rely on the concept of practical identifiability and profile likelihoods that fully account for nonlinearity. So far, the literature on profile likelihoods appears to mainly concentrate on (finite dimensional) statistics, as well as applications in systems biology, geography and econometrics. To the best of our knowledge, their use in inverse problems involving models in infinite dimensional spaces, especially to parameter identification problems in PDEs, has not been investigated yet.

In a statistical framework, parameter and prediction uncertainties can be quantified in terms of confidence and credible intervals. Confidence and credible intervals capture the range of plausible parameter and model predictions in accordance with a predefined statistical measure, e.g., the likelihood ratio. For the construction of confidence and credible intervals, local approximations \cite{Meeker1995,MurphyVaa2000}, bootstrapping \cite{Joshi2006}, Bayesian methods \cite{Wilkinson2007} and profile likelihoods \cite{MurphyVaa2000} are employed. Local approximation such as the Wald approximation \cite{Meeker1995} and the Fisher Information Matrix (FIM) based approximation \cite{MurphyVaa2000} are computationally efficient but merely provide rough estimates of confidence intervals. Bootstrapping provides non-local estimates but should only be applied to models without practical non-identifiablities \cite{FroehlichThe2014}. Bayesian methods and profile likelihoods appear to be most reliable and consistent \cite{RaueKre2013,HugRau2013,HrossHas2016}.

Bayesian methods construct representative samples from the posterior distribution, thereby assessing the uncertainty of all parameters and model predictions simultaneously
\cite{KaipioSomersalo}. Profile likelihood methods explore the uncertainty of individual parameters and model predictions using repeated local optimizations. The credible intervals computed using Bayesian methods employ marginalization, while confidence intervals computed using profile likelihoods rely on maximum projections. Raue et al. \cite{RaueKre2013} demonstrated the latter can be advantageous as the coverage of regions with high likelihood values is ensured. In addition, the calculation of profile likelihoods tends to be computationally more tractable than the sampling of the posterior distribution \cite{RaueKre2013,HugRau2013,HrossHas2016}. This also holds if sophisticated sampling procedures \cite{HaarioLai2006,GirolamiCal2011,RigatMir2012} are used. Nevertheless, for computationally demanding problems, also the application of classical profile likelihood methods is prohibitive \cite{HrossHas2016,HockHas2013,Lockley2015}.

%To improve the computational efficiency of profile likelihood calculations, Chen and Jennrich \cite{ChenJen2002} proposed a dynamical system which can be employed to determine profile likelihoods by numerical integration. This integration based method exploits gradient and hessian of the objective function to directly determine the update direction for the parameters. The resulting path in parameter space describes the profile likelihoods. Chen and Jennrich \cite{ChenJen2002} obtained promising results for simple likelihood functions. Recently also applications to ordinary differential equation models has been discussed \cite{KreutzRau2013}. 

To improve the computational efficiency of profile likelihood calculations, Chen and Jennrich \cite{ChenJen2002} proposed an integration based approach. This approach relies on a differential algebraic equation (DAE) which evolves along the profile likelihood. The trajectories of this systems provide the parameter profile without the need for repeated optimization. Mass matrix and vector field of the DAE are computed from the gradient and hessian of the objective function. Chen and Jennrich \cite{ChenJen2002} obtained promising results for simple likelihood functions. In the last years also the application to ordinary differential equation (ODE) models has been discussed \cite{KreutzRau2013}. 

In this paper, we will generalize integration based profile likelihood calculation to PDE constrained parameter estimation problems. We will introduce a reduced and a full formulation for a statistically motivated objective function and discuss their properties. As the calculation of the Hessian is potentially computationally intensive, approximation will be considered and combined with a retraction term. The different approaches will be illustrated and evaluated using an example from systems biology.

The remainder of this paper is organized as follows: In Section~\ref{sec:Mathematical model} we will introduce the considered class of mathematical models and observation operator. The parameter estimation problem and uncertainty analysis using profile likelihoods will be outlined in Section~\ref{sec:Parameter optimization and uncertainty analysis}. In Section~\ref{sec:Profile likelihood calculation for the reduced problem} and~\ref{sec:Profile likelihood calculation for the full problem} the integration based profile likelihood calculation for the reduced and the full problem are presented. The relation of these two approaches is discussed in Section~\ref{sec:Comparison of full and reduced formulation of integration based profile likelihood calculation}. The proposed integration based profile likelihood calculation for PDE models is evaluated in Section~\ref{sec:Numerical evaluation of integration based profile likelihood calculation} for a model of gradient formation in fission yeast. The paper concludes with a discussion of the results and an outlook in Section~\ref{sec:Conclusion}.

%%%%
%%%%
%%%%
\section{Mathematical model}
\label{sec:Mathematical model}

We consider parameter estimation in partial differential equation models
\begin{equation}\label{eq_Model}
\eqalign{
&u_t + C(\theta,u) = f(\theta)   \mbox{ in } ]0,T[\cr
&u(0) = u_0,}
\end{equation}
with state variable $u \in V$, defined over a spatial domain, and parameter vector $\theta \in \Theta \subseteq \mathbb{R}^{n}$. The operator $C(\theta,.):V\rightarrow V^\ast$, mapping from a separable, reflexive Banach space $V$ into its dual $V^\ast$, is equipped with appropriate boundary conditions, where $V \subset H \cong H^\ast \subset V^\ast$ is a Gelfand triple such that $V$ is imbedded continuously and densely into a Hilbert space $H$. To guarantee the existence of a weak solution $u\in W(0,T)=L^2(0,T;V)\cap H^1(0,T;V^\ast)$ of~(\ref{eq_Model}), according to (\cite{Zeidler90}, p. 770 ff.), we assume that the operator $C$ meets the following assumption:
\begin{assumption}[Existence of a weak solution] \label{Ass:weak_solution}~
\begin{itemize}
\item $u_0 \in H$ and $f(\theta)\in L^2([0,T]; V^\ast)$ are given. 
\item $C(\theta,.)$ is monotone and hemicontinuous. 
\item $C(\theta,.)$ is coercive, i.e. there exist $c_0$ and $c_1$ such that $\left\langle C(\theta,u),u\right\rangle_{V^\ast,V} \geq c_0\|u\|^2_V-c_1$.
\item $C(\theta,.)$ satisfies the growth condition, i.e. there exists a nonnegative function $c_2\in L^2(0,T)$ and a constant $c_3>0$, such that
$\|C(\theta,u)(t)\|_{V^\ast}\leq c_2(t)+c_3\|u(t)\|_V$ for all $u \in V$ and $t\in ]0,T[$.
\item The function $t\mapsto \left\langle C(\theta,u)(t),v\right\rangle_{V^\ast,V}$ is measurable on $]0,T[$ for all $u,v\in V$.
\end{itemize} 
\end{assumption}
Assumption~\ref{Ass:weak_solution} holds for models from a broad range of applications \cite{Troeltzsch2010} and ensures the existence of a parameter-to-state map 
\[\eqalign{
S:\R^{n_\theta}\to W(0,T), \mbox{ with } 
u=S(\theta) \mbox{ solving }(\ref{eq_Model})}.
\]
As the measurement of $u$ can be limited by experimental technologies, we consider potentially partial observations,
\begin{equation}
y = Q(\theta,u).
\label{eq:noise}
\end{equation}
The observation operator $Q(\theta,.):W(0,T)\rightarrow \R^{K}$ maps $u$ onto the observation $y \in \R^{K}$. The observation $y$ is a collection of different scalar observables measured at different time points. The index $k$ enumerates all the combinations of observables and time points, $y_k = (Q(\theta,u))_{k}=Q_{k}(\theta,u)$ for $k = 1,\ldots, K$.

In practice the observations are corrupted by measurement noise. The noise corrupted measurement of the observable $y_k$ is denoted by $\overline{y}_k$. For additive, normally distributed measurement noise it holds that
\begin{equation}
\overline{y}_k = y_k + \eps_k \mbox{ with } \eps_k \sim \mathcal{N}(0,\sigma_k^2(\theta)).
\label{eq:noise_1}
\end{equation}
The parameters of the noise model, here the variance $\sigma_k^2(\theta)$, are potentially unknown and can depend on the parameter.

%%%%
%%%%
%%%%
\section{Parameter estimation problem and uncertainty analysis}
\label{sec:Parameter optimization and uncertainty analysis}

We estimate the parameters using a likelihood-based approach. The likelihood is the conditional probability of observing the measured data $\overline{y}_k$, $k = 1,\ldots,K$, given the parameter vector $\theta$. The likelihood of observing the measured data depends implicitly on the noise model and is e.g. for additive, normally distributed measurement noise given by
\[\mathrm{L}(\theta)=\prod_{k=1}^{K} \frac{1}{\sqrt{2\pi}\sigma_{k}(\theta)}\exp\left(-\frac{1}{2} \left(\frac{\overline{y}_{k}-Q_{k}(\theta,S(\theta))}{\sigma_{k}(\theta)}\right)^2\right).\]

\begin{remark}
The methods and results we will present do not assume a particular noise model or likelihood function. We merely assume that the likelihood function is twice continuously differentiable.
\end{remark}

The parameter vector $\hat\theta$ which maximizes $\mathrm{L}(\theta)$ is the maximum likelihood estimate. To improve the numerical evaluation and the optimizer convergence, the maximum likelihood estimate is usually determined by minimizing the negative log-likelihood function, 
\begin{equation} \label{eq: reduced objective function}
\mathbf{j}(\theta) = - \log \mathrm{L}(\theta) = \frac{1}{2} \sum^{K}_{k=1} \left(\log\left(2\pi \sigma^2_{k}(\theta)\right) + \left(\frac{\bar{y}_{k}-Q_{k}(\theta,S(\theta))}{\sigma_{k}(\theta)}\right)^2\right),
\end{equation}
with its non-reduced counterpart
\begin{equation*}
\mathbf{J}(\theta,u)= \frac{1}{2} \sum^{K}_{k=1} \left(\log\left(2\pi \sigma^2_{k}(\theta)\right) + \left(\frac{\bar{y}_{k}-Q_{k}(\theta,u)}{\sigma_{k}(\theta)}\right)^2\right).
\end{equation*}
This yields the PDE constrained optimization problem 
\begin{equation}
\eqalign{
& \min_{\theta \in\Theta, u\in W(0,T)}  \mathbf{J}(\theta,u) \cr
&\mbox{s.t. }u_t + C(\theta,u) = f(\theta)   \mbox{ in } ]0,T[\cr
&\hspace{2.1cm} u(0) = u_0.
\label{min_reduced}}
\end{equation}
The maximum likelihood estimate $\hat\theta$, i.e. the optimum of~(\ref{min_reduced}), is potentially non-unique and might strongly depend on the measurement noise. %Accordingly, the reliability of the parameter estimates need to be assessed.
To assess the parameters and prediction uncertainties we consider confidence regions and confidence intervals.
\begin{dfn}[Confidence region]~\newline
For the parameter vector $\theta\in\Theta$ we define the confidence region to the confidence level $\alpha$ as
\[\eqalign{
\mathrm{CR}_{\alpha} 
&=  \left\lbrace\theta\in\Theta\left|\frac{\mathrm{L}(\theta)}{\mathrm{L}(\hat{\theta})}\geq\exp\left(-\frac{\Delta_{\alpha}}{2}\right)\right.\right\rbrace, \cr
&=  \left\lbrace\theta\in\Theta\left|2\left(\mathbf{j}(\theta)-\mathbf{j}(\hat{\theta})\right) \leq \Delta_{\alpha}\right.\right\rbrace,
}\]
with $\Delta_{\alpha}$ denoting the $\alpha$th-percentile of the $\chi^2$ distribution with one degree of freedom. 
\end{dfn}
From the confidence regions, the confidence intervals for individual model properties $\mathbf{G}(\theta,u)$, with $\mathbf{G}: \Theta \times V \mapsto \mathbb{R}$, can be derived. The reduced form of $\mathbf{G}(\theta,u)$ is denoted by $\mathbf{g}(\theta) = \mathbf{G}(\theta,S(\theta))$, with $\mathbf{g}: \Theta \mapsto \mathbb{R}$. Model properties are for instance individual parameters, functions of parameters or properties of the solution of the model.
\begin{dfn}[Confidence interval]~\newline
The confidence interval for a model property is the projection of $\mathrm{CR}_\alpha$ onto $\mathbf{g}(\theta)$,
\begin{equation}
\mathrm{CI}_{\alpha,\mathbf{g}(\theta)} = P_{\mathbf{g}(\theta)} \mathrm{CR}_{\alpha} = \left\lbrace c\left|\exists \theta \in \mathrm{CR}_{\alpha} \wedge \mathbf{g}(\theta) = c\right.\right\rbrace
\label{eq:conf interval}
\end{equation}
\end{dfn}
The evaluation of the confidence region requires the calculation of level sets of likelihood functions. For problems with $n_\theta \gg 1$ this is non-trivial. To determine confidence intervals without calculating confidence regions, profile likelihoods \cite{MurphyVaa2000} can be used. The profile likelihood for a scalar function $\mathbf{g}(\theta)$, $\mathrm{PL}_{\mathbf{g}(\theta)}(c)$, is the maximal likelihood value for $\mathbf{g}(\theta) = c$ \cite{ChenJen1996}.
\begin{dfn}[Profile likelihood]~\newline
For the scalar function $\mathbf{g}(\theta)$ we define the profile likelihood as
\begin{equation}
\mathrm{PL}_{\mathbf{g}(\theta)}(c) =\max_{\theta\in\Theta} \mathrm{L}(\theta) \mbox{ subject to } \mathbf{g}(\theta) = c.
\label{eq:Profile}
\end{equation}
For values $c$ outside the range of $\mathbf{g}(\theta)$, $\mathrm{PL}_{\mathbf{g}(\theta)}(c) = 0$.
\end{dfn}
In other words the profile likelihood provides the maximum projection of the likelihood along $\mathbf{g}(\theta)$. Accordingly, the confidence interval for $\mathbf{g}(\theta)$ follows from (\ref{eq:conf interval}) as
\[\mathrm{CI}_{\alpha,\mathbf{g}(\theta)} =  \left\lbrace c\left|\frac{\mathrm{PL}_{\mathbf{g}(\theta)}(c)}{\mathrm{L}(\hat{\theta})} \geq\exp\left(-\frac{\Delta_{\alpha}}{2}\right)\right.\right\rbrace.\]
The relation among likelihood function, confidence region/intervals and profile likelihoods is illustrated in Figure~\ref{fig:Unc_schematic}. Note that the confidence intervals for the individual parameters are obtained by the projection of the confidence region as well as by thresholding the profile likelihood. For $\mathbf{g}(\theta) = \theta_j$ this provides the confidence interval of parameter $\theta_j$, while for other choices of $\mathbf{g}(\theta)$ more involved parameter dependent model properties can be assessed, e.g. the product of two parameters.

The confidence intervals to a confidence level $\alpha$ can be bounded or unbounded:
\begin{dfn}[Practical identifiability]~\newline
A model property $\mathbf{g}(\theta)$ is called practically identifiable if its confidence interval $\mathrm{CI}_{\alpha,\mathbf{g}(\theta)}$ is bounded; otherwise it is called practically non-identifiable.
\end{dfn}

\begin{figure}
\centering
\includegraphics[width=\textwidth]{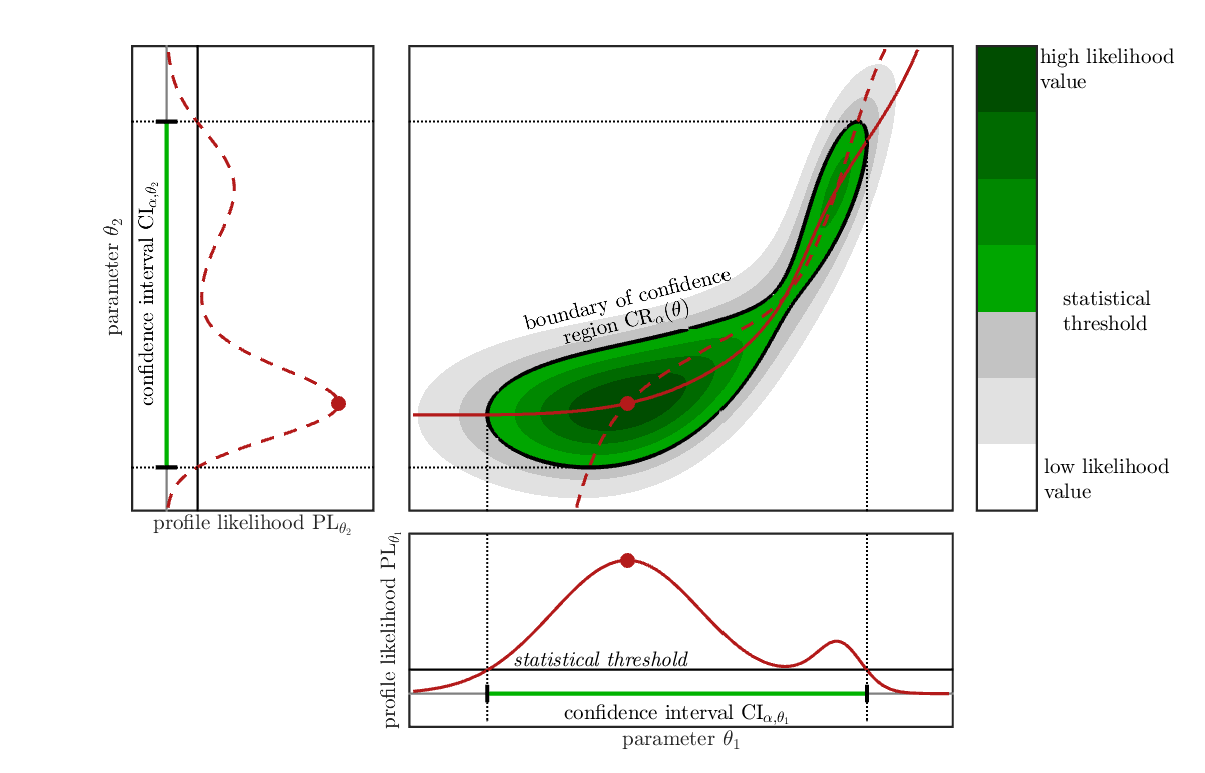}
\caption{\textbf{Illustration of confidence region, confidence intervals, profile likelihoods and their relation.} (big panel) Likelihood function landscape (shading), confidence region (\textcolor{RegGreen}{\rule{0.25cm}{7pt}}) and profile likelihood path $\theta_c$ ($\theta_1$:\textcolor{ProfRed}{\rule[.8mm]{0.25cm}{2pt}};$\theta_2$:\textcolor{ProfRed}{\rule[.8mm]{0.08cm}{2pt}}\rule[.8mm]{0.08cm}{0pt}\textcolor{ProfRed}{\rule[.8mm]{0.08cm}{2pt}}). (small panels) Profile likelihood ratio (\textcolor{ProfRed}{\rule[.8mm]{0.25cm}{2pt}}) and confidence interval (\textcolor{ConfGreen}{\rule[.8mm]{0.25cm}{2pt}}) for $\theta_1$ and $\theta_2$. The relation of different quantities is indicated using dotted lines. The significance threshold is indicated in all three figures as solid black line.}
\label{fig:Unc_schematic}
\end{figure}

%A particular value of $\theta_i$ can be rejected, if the profile likelihood $\mathrm{PL}(\theta_i)$ is low compared the the likelihood $p(\mathcal{D}|\theta^*)$ at the globally optimal parameter point $\theta^*$.

%%%%%%
%%%%%%
%%%%%%
\section{Profile likelihood calculation for the reduced problem}
\label{sec:Profile likelihood calculation for the reduced problem}

In this section, we introduce optimization and integration based profile likelihood calculation for the reduced form of PDE constrained optimization problems. The formulation of the integration based  profile likelihood calculation method is adapted from the results of Chen and Jennrich \cite{ChenJen2002}. In particular we establish its validity for function spaces. For simplicity of exposition we consider the case without constraints on the parameters, i.e. $\theta \in\Theta = \R^n$.

Since our analysis will rely on differentiation of the first order necessary optimality conditions, we will make the following assumptions on smoothness of the involved functions

\begin{assumption}\label{Ass:diff}
$C\colon \R^{n_\phi}\times V \rightarrow V^\ast$, $f\colon\R^{n_\phi}\rightarrow V^\ast$ and $J\colon \R^{n_\theta}\times W(0,T)\rightarrow \R$ are twice continuously differentiable.
\end{assumption}

\begin{figure}
\centering
\includegraphics[width=\textwidth]{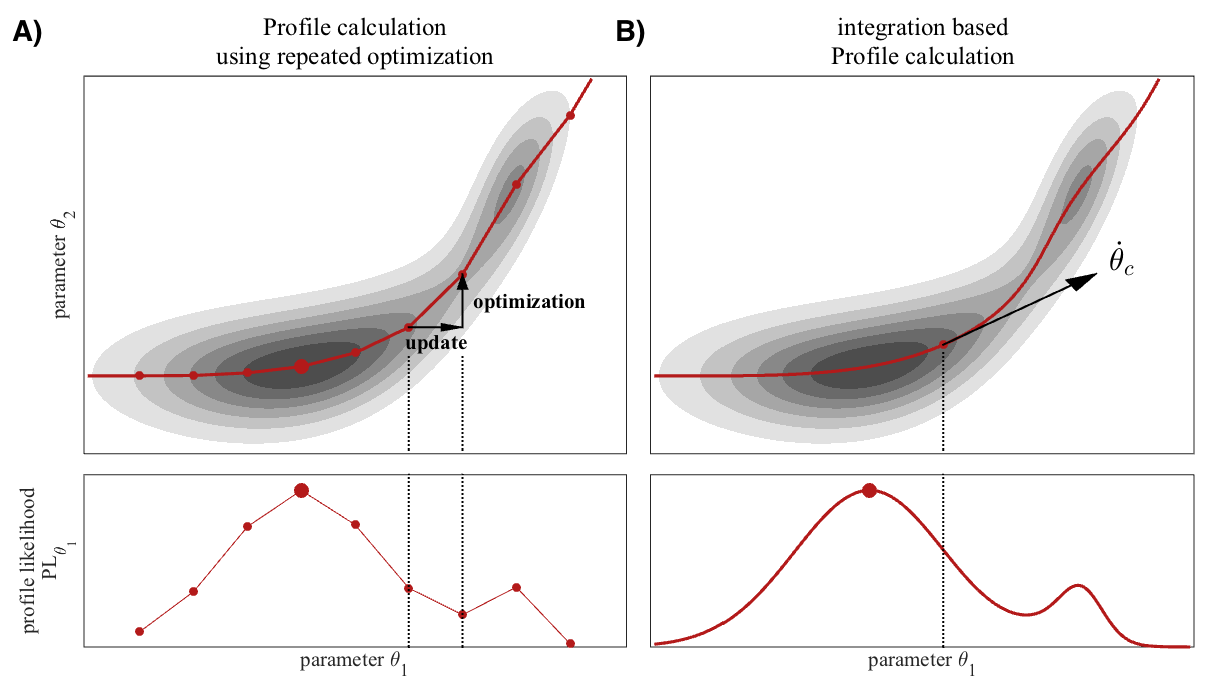}
\caption{\textbf{Illustration of optimization based profile likelihood calculation (upper panels) and integration based profile likelihood calculation (lower panels).} (upper panel in A) Optimization based profile likelihood calculation for likelihood function (shading) using update and re-optimization step (arrows). (upper panel in B) Evaluation points (\textcolor{ProfRed}{$\bullet$}) and approximation of the profile likelihood (\textcolor{ProfRed}{\rule[.8mm]{0.25cm}{2pt}}). (lower panel in A) Integration based profile likelihood calculation for likelihood function (shading) using continuous system with derivative (arrow) tangential to the parameter trajectory (\textcolor{ProfRed}{\rule[.8mm]{0.25cm}{2pt}}). (lower panel in B) Profile likelihood (\textcolor{ProfRed}{\rule[.8mm]{0.25cm}{2pt}}) obtained using integration based method.
}
\label{fig:Prof_schematic}
\end{figure}

%%%%%%
%%%%%%
\subsection{Optimization based profile likelihood calculation}
\label{sec: Optimization based profile likelihood calculation}

Optimization based methods approximate the profile likelihood $\mathrm{PL}_{\mathbf{g}(\theta)}(c)$ by evaluating it on a grid $\{c_l\}_l$ (Figure~\ref{fig:Prof_schematic} upper panels). For each point $c$ the reduced negative log-likelihood function is minimized,
\begin{equation}
\min_{\theta\in\Theta} \mathbf{j}(\theta) \mbox{ subject to } \mathbf{g}(\theta) = c.
\label{eq:PL}
\end{equation}
This minimization yields the optimal parameter vector, $\theta_c := \hat\theta(c)$, and the corresponding value of the negative log-likelihood function, $\mathbf{j}(\theta_c)$. It holds that $\mathrm{PL}_{\mathbf{g}(\theta)}(c) =\exp(-\mathbf{j}(\theta_c))$.

State-of-the-art methods construct the grid $\{c_l\}_l$ iteratively, starting at the optimal parameter vector $\hat\theta$ of~(\ref{min_reduced}) with $\hat c = \mathbf{g}(\hat\theta)$ \cite{Raue2009}. An iteration consists of two steps: (i) the update of the constraint $c_l$ using an adaptive approach; and (ii) the local optimization of the parameters. The adaptation controls the change in the objective function, $\mathbf{j}(\theta_{c_{l-1}})$ to $\mathbf{j}(\theta_{c_l})$, and the number of iterations. The starting point $\theta_{c_l}^{(0)}$ of the local optimization for $c_l$ is constructed from previous points. Most implementations use as starting point
\begin{enumerate}
\item \textbf{0th order proposal:} the optimal point for $c_{l-1}$, $\theta_{c_l}^{(0)} = \theta_{c_{l-1}}$, or
\item \textbf{1st order proposal:} the linear extrapolation based on the optimal points for $c_{l-1}$ and $c_{l-2}$,
$$\theta_{c_l}^{(0)} = \theta_{c_{l-1}} + \frac{c_{l}-c_{l-1}}{c_{l-1}-c_{l-2}} (\theta_{c_{l-1}} - \theta_{c_{l-2}}).$$
\end{enumerate}
The 0th order proposal is illustrated in Figure~\ref{fig:Prof_schematic} A (upper panel). In practice the 1st order proposal, which uses additional topological information, yields starting points which are closer to the optimum $\theta_c$. Accordingly, this approach tends to be computationally more efficient.

Optimization-based profile likelihood calculation is computationally efficient compared to other uncertainty analysis methods~\cite{RaueKre2013,HugRau2013,HrossHas2016}. It, however, becomes computationally demanding if the number of necessary iterations increases or an individual local optimization is computationally expensive. The number of iterations is influenced by the structure of the objective function landscape, e.g., non-identifiabilities. The computational complexity of the local optimization is determined by the computation time of the forward problems (and its derivatives). Both are issues for a range of practical applications including PDE constrained problems~\cite{HockHas2013}. 

%%%%%%
%%%%%%
\subsection{Integration based profile likelihood calculation}
\label{sec: Integration based profile likelihood calculation - reduced}

Integration based profile likelihood calculation addresses the drawbacks of optimization based methods by exploiting the differential geometry of the reduced optimization problem~\cite{ChenJen2002}. This is achieved by considering the Lagrange function for~(\ref{eq:PL}),
\begin{equation*}
\ell(\theta) = \mathbf{j}(\theta) + \lambda (\mathbf{g}(\theta) - c),
\end{equation*}
in which $\lambda \in \R$ denotes the Lagrange multiplier. From the Lagrange function the first order optimality conditions,
\begin{equation}
\eqalign{
\nabla_{\theta} \mathbf{j}(\theta) + \lambda\nabla_{\theta}\mathbf{g}(\theta) = 0\\
\mathbf{g}(\theta) = c,}
\label{eq:PL_diff_red}
\end{equation}
can be derived. This system of equations describes the dependence of the minimizing parameter vector $\theta$ and the Lagrange multiplier $\lambda$ on $c$. Therefore we use the notation $\theta_c := \theta(c)$ and $\lambda_c := \lambda(c)$. The differentiation of~(\ref{eq:PL_diff_red}) with respect to $c$ yields an evolution equation for the pair $(\theta_c,\lambda_c)$, 
\begin{equation}
\underbrace{\left(\begin{array}{cc}
\nabla_{\theta}^2 \mathbf{j}(\theta_c) + \lambda_c\nabla_{\theta}^2 \mathbf{g}(\theta_c) & \nabla_{\theta} \mathbf{g}(\theta_c)\cr
\nabla_{\theta} \mathbf{g}(\theta_c)^T & 0\end{array}\right)}_{:= M_\mathrm{red}(\theta_c)} \left(\begin{array}{c}\dot{\theta}_c\cr \dot{\lambda}_c\end{array}\right)=\left(\begin{array}{c}0\cr 1\end{array}\right),
\label{eq:PL_diff_red2}
\end{equation}
where $\dot{\theta}_c$ and $\dot{\lambda}_c$ are derivatives with respect to $c$. The solution of the differential algebraic equation (DAE)~(\ref{eq:PL_diff_red2}) for a starting point which solves~(\ref{eq:PL}) for $c = c_0$ yields the profile $\theta_c$ for $c\in[c_0,c_\mathrm{end}]$.

\begin{proposition}
Let $\mathbf{j}:\R^n\to\R$ and $g:\R^n\to\R$ be twice continuously differentiable and let $({\theta}_c,\lambda_c)_{c\in[c_0,c_1]}$ be a solution of (\ref{eq:PL_diff_red2}) with initial data $(\theta_{c_0},\lambda_{c_0})$ solving (\ref{eq:PL_diff_red}) for $c=c_0$.

Then for all $c\in[c_0,c_1]$, $(\theta_c,\lambda_c)$ solves the optimality conditions (\ref{eq:PL_diff_red}).
\end{proposition}
\begin{proof}
For any fixed $c_1>c_0$ we define $\Psi:[c_0,c_1]\to\R^{n+1}$ by $\Psi(c)=(\theta_c,\lambda_c)$ 
and $\Phi:\R^{n+1}\to\R^{n+1}$ by $\Phi(\theta,\lambda)=
\left(\nabla_{\theta} \mathbf{j}(\theta) + \lambda\nabla_{\theta} \mathbf{g}(\theta),
\mathbf{g}(\theta)\right)^T$,
so that we can rewrite (\ref{eq:PL_diff_red}) for $c\in[c_0,c_1]$ as 
\begin{equation*}
\Phi(\Psi(c))-\left(\begin{array}{c}0\cr c\end{array}\right)=0 \quad \forall c\in[c_0,c_1]\,.
\end{equation*}
Under the differentiability assumptions made here this is equivalent to 
\begin{equation*}
\hspace{-2cm}
\Phi(\Psi(c_0))-\left(\begin{array}{c}0\cr c_0\end{array}\right)=0\ \mbox{ and } \
\frac{d \Phi }{d (\theta,\lambda)}\left(\Psi(c)\right)\dot\Psi(c)-\left(\begin{array}{c}0\cr 1\end{array}\right)=0 \quad \forall c\in[c_0,c_1]\,,
%\Phi(\Psi(c_0))-(0, c_0)^T=0\ \mbox{ and } \
%\Phi'(\Psi(c))\Psi'(c)-(0, 1)^T=0 \quad \forall c\in[c_0,c_1]\,,
\end{equation*}
i.e., (\ref{eq:PL_diff_red2}).
\end{proof}
The trajectory $\theta_c$ of~(\ref{eq:PL_diff_red2}) is the path in parameter space along which the minimum of the constrained optimization problem~(\ref{eq:PL}) is attained. The evaluation of the objective function $\mathbf{j}(\theta_c)$ along this trajectory yields the profile likelihood $\mathrm{PL}_{\mathbf{g}(\theta)}(c) =\exp(-\mathbf{j}(\theta_c))$. Accordingly, the profile likelihood can be computed without optimization. Instead, the update directions are determined by the derivatives of $\mathbf{g}(\theta)$ and $\mathbf{j}(\theta)$.

The numerical integration of~(\ref{eq:PL_diff_red2}) relies on the evaluation of the matrix-vector product $M_\mathrm{red}(\theta_c) (\dot\theta_c,\dot\lambda_c)^T$. This matrix-vector product contains the terms $(\nabla_\theta \mathbf{j}(\theta_c)  + \lambda_c\nabla_\theta^2 \mathbf{g}(\theta_c)) \dot\theta_c$, $\nabla_\theta \mathbf{g}(\theta_c) \dot\lambda_c$ and $\nabla_\theta g^T(\theta_c) \dot\theta_c$. As $\mathbf{g}(\theta)$ and $\mathbf{j}(\theta)$ are functions of the PDE solution, their derivatives depend on the parameter-to-state mapping $S(\theta)$. These derivatives with respect to $\theta$ are explicitly given as 
\begin{equation}
\eqalign{
\fl \mathbf{g}_{\theta_i}(\theta_c)&=\mathbf{G}_{\theta_i}(\theta_c, S(\theta_c))+\mathbf{G}_{u}(\theta_c,S(\theta_c))S_{\theta_j}(\theta_c), \\
\fl \mathbf{g}_{\theta_i \theta_j}(\theta_c)&=\mathbf{G}_{\theta_i \theta_j}(\theta_c, S(\theta_c))+\mathbf{G}_{\theta_i u}(\theta_c,S(\theta_c))S_{\theta_j}(\theta_c)+\mathbf{G}_{u \theta_j}(\theta_c,S(\theta_c))S_{\theta_i}(\theta_c) \\
\fl&+\mathbf{G}_{u u}(\theta_c,S(\theta_c))S_{\theta_i}(\theta_c)S_{\theta_j}(\theta_c)+ \mathbf{G}_{u}(\theta_c,S(\theta_c))S_{\theta_i \theta_j}(\theta_c),
}
\end{equation}
and
\begin{equation} \label{second_deriv_reduced}
\eqalign{
\fl \mathbf{j}_{\theta_i \theta_j}(\theta_c)&=\mathbf{J}_{\theta_i \theta_j}(\theta_c, S(\theta_c))+\mathbf{J}_{\theta_i u}(\theta_c,S(\theta_c))S_{\theta_j}(\theta_c)+\mathbf{J}_{u \theta_j}(\theta_c,S(\theta_c))S_{\theta_i}(\theta_c) \\
\fl&+\mathbf{J}_{u u}(\theta_c,S(\theta_c))S_{\theta_i}(\theta_c)S_{\theta_j}(\theta_c)+ \mathbf{J}_{u}(\theta_c,S(\theta_c))S_{\theta_i \theta_j}(\theta_c).
}
\end{equation}
Here $\mathbf{G}$ and $\mathbf{J}$ denote the unreduced form of the model property and the objective function. The sensitivities of the parameter-to-state map $S(\theta)$ can be calculated using forward sensitivity equations derived by differentiating~(\ref{eq_Model}) for $u = S(\theta_c)$. This differentiation yields the first and second order derivatives, $e=S_{\theta_i}(\theta)$ and $z=S_{\theta_i\theta_j}(\theta)$, for $i,j=1,\ldots,n$, as solutions to,
\begin{equation}
\label{eq:1st order forward sensitivities}
\hspace{-2.5cm} \left\{
\begin{array}{l}
e_t+C_u(\theta,S(\theta)) e = f_{\theta_i}(\theta)-C_{\theta_i}(\theta,S(\theta))_=: -h^\theta_i\\
e(0)=0
\end{array}\right.
\end{equation}
and
\begin{equation}
\label{eq:2nd order forward sensitivities}
\hspace{-2.5cm} \left\{
\begin{array}{ll}
z_t+C_u(\theta,S(\theta)) z=&f_{\theta_i \theta_j}(\theta)-C_{\theta_i\theta_j}(\theta,S(\theta))-C_{\theta_i u}(\theta,S(\theta))S_{\theta_j}(\theta)\\
&-C_{u\theta_j}(\theta,S(\theta)) S_{\theta_i}(\theta)-C_{uu}(\theta,S(\theta)) S_{\theta_i}(\theta)S_{\theta_j}(\theta)=:-h^\theta_{ij}
\\
z(0)=0,
\end{array}\right.
\end{equation}
In addition to solving the nonlinear PDE~(\ref{eq_Model}), this requires the solutions of $n_\theta$ linear PDEs for the first order sensitivities and $n_\theta (n_\theta+1) / 2$ linear PDEs for the second order sensitivities. An alternative to forward sensitivities is the evaluation of the aforementioned Hessian and gradient vector product by adjoint methods. These enable the computation of the objective function gradient by solving just one linearized PDE and computation of the Hessian of the objective function by solving two additional linearized PDEs. Namely, defining $p$ as the solution of the adjoint equation
\begin{equation}
\left\{
\begin{array}{l}
p_t-C_u(\theta,S(\theta))^* p= \mathbf{J}_u(\theta,S(\theta))  \\
p(T)=0
\end{array}\right.
\label{eq:adjoint0}
\end{equation}
(see also (\ref{eq:adjoint}) below) and using the fact that $u=S(\theta)$ solves the PDE~(\ref{eq_Model}) followed by integration by parts, we obtain
\[
\hspace{-2.5cm}\eqalign{
\frac{\partial \mathbf{j}}{\partial\theta_i}(\theta)
= \frac{\partial }{\partial\theta_i} \Bigl(
\mathbf{J}(\theta,S(\theta))+\int^{T}_0{\left\langle S(\theta)_t+ C(\theta,S(\theta))-f(\theta),p\right\rangle_{V^\ast,V}}dt
\Bigr)\\
= \mathbf{J}_{\theta_i} (\theta,S(\theta))+ 
%\int^{T}_0{\left\langle\mathbf{J}_u(\theta,S(\theta)),S_{\theta_i}(\theta)\right\rangle_{V,V^\ast}}dt
\mathbf{J}_u(\theta,S(\theta))S_{\theta_i}(\theta)\\
\quad+\int^{T}_0{\Bigl(\left\langle S_{\theta_i}(\theta)_t+ C_{\theta_i}(\theta,S(\theta))+ C_u(\theta,S(\theta))S_{\theta_i}(\theta)-f_{\theta_i}(\theta),p\right\rangle_{V^\ast,V}\Bigr)}dt\\
%= \mathbf{J}_{\theta_i} (\theta,S(\theta))+\int^{T}_0{\Bigl(\left\langle -p_t+C_u(\theta,S(\theta))^* p+ \mathbf{J}_u(\theta,S(\theta)),S_{\theta_i}(\theta)\right\rangle_{V,V^\ast}-\left\langle f_{\theta_i}(\theta),p\right\rangle_{V^\ast,V}\Bigr)}dt\\
= \mathbf{J}_{\theta_i} (\theta,S(\theta))+
\int^{T}_0{\left\langle C_{\theta_i}(\theta,S(\theta))-f_{\theta_i}(\theta),p\right\rangle_{V^\ast,V}}dt\,.
}
\]
To calculate the Hessian-vector product $\nabla^2\mathbf{j}(\theta)\zeta$ for some $\zeta\in\R^{n_\theta}$,
we apply the same procedure to the auxiliary minimization problem $\min_{\theta\in\Theta} \nabla\mathbf{j}(\theta)^T\zeta$, which is then equivalent to 
\[
\eqalign{
\min_{\theta\in\Theta, (u,p)\in W(0,T)^2}  \tilde{\mathbf{J}}(\theta,(u,p))& \cr
\mbox{s.t. } u_t+C(\theta,u) = f(\theta)\qquad &u(0)=u_0\cr
\hspace*{0.75cm}p_t-C_u(\theta,u)^* p= \mathbf{J}_u(\theta,u)  \qquad &p(T)=0
}
\]
with 
\[
\tilde{\mathbf{J}}(\theta,(u,p))=\zeta^T\nabla_\theta\mathbf{J}(\theta,u)
+\int_0^T\langle \zeta^T\nabla_\theta C(\theta,u)-\zeta^T\nabla_\theta f(\theta),p\rangle_{V^\ast,V}dt\,.
\]
Defining $v,w$ as the solutions of 
\[
\hspace{-2.5cm}\eqalign{
v_t+C_u(\theta,S(\theta))v = -\zeta^T\nabla_\theta C(\theta,S(\theta))+\zeta^T\nabla_\theta f(\theta)    
\quad &v(0)=0\cr
w_t-C_u(\theta,S(\theta))^* w= (B(\theta,S(\theta),P(\theta))^*v+(\zeta^T\nabla_\theta C_u(\theta,S(\theta)))^*P(\theta)\\
\hspace*{4.2cm}+\zeta^T\nabla_\theta\mathbf{J}_u(\theta,S(\theta))+\mathbf{J}_{uu}(\theta,S(\theta))^*v 
\quad &w(T)=0}
\]
where we define $B$ by 
\[
\langle B(\theta,a,b)c,d\rangle_{V^\ast,V}=\langle C_{uu}(\theta,a)(c,d),b\rangle_{V^\ast,V} \mbox{ for all } \theta\in\Theta\,, \ a,b,c,d\in V
\]
and $P(\theta)=p$ as the solution to (\ref{eq:adjoint0}), we arrive at 
\[
\hspace{-2.5cm}\eqalign{
(\nabla_\theta^2 \mathbf{j}(\theta) \zeta)_i= \frac{\partial }{\partial\theta_i} \tilde{\mathbf{j}}(\theta)
= \frac{\partial }{\partial\theta_i} \tilde{\mathbf{J}}(\theta,(S(\theta),P(\theta)))\\
= \frac{\partial }{\partial\theta_i}\Bigl(
\zeta^T\nabla_\theta\mathbf{J}(\theta,S(\theta))+\int^{T}_0{\left\langle \zeta^T\nabla_\theta C(\theta,S(\theta))-\zeta^T\nabla_\theta f(\theta),P(\theta)\right\rangle_{V^\ast,V}}dt\\
\qquad\qquad + \int_0^T\langle S(\theta)_t+C(\theta,S(\theta))-f(\theta),w\rangle_{V^\ast,V}dt\\
\qquad\qquad + \int_0^T\langle -P(\theta)_t+C_u(\theta,S(\theta))^*P(\theta)+\mathbf{J}_u(\theta,S(\theta)),v\rangle_{V^\ast,V}dt
\Bigr)\\
=(\nabla_\theta^2 \mathbf{J}(\theta,S(\theta)) \zeta)_i + \mathbf{J}_{\theta_i u}(\theta,S(\theta))v
+ \int_0^T\Bigl(\langle C_{\theta_i}(\theta,S(\theta))-f_{\theta_i}(\theta),w\rangle_{V^\ast,V}\\
\qquad+\langle ((\nabla_\theta^2 C(\theta,S(\theta))-\nabla_\theta^2 f(\theta))\zeta)_i+C_{\theta_i u}(\theta,S(\theta))v,P(\theta)\rangle_{V^\ast,V}
\Bigr)dt\,.
}
\]

The numerical simulation of~(\ref{eq:PL_diff_red2}) with explicit or implicit time stepping can introduce numerical errors, which results in a divergence of the trajectory from the profile path and leads to an underestimation of the profile likelihood. This effect can be counterbalanced by the incorporation of a retraction term, which results in a minimization of $\mathbf{j}(\theta)$ for the given constraint, 
\begin{equation} \label{eq: reduced integration-based profile calculation, stabilized}
\left(\begin{array}{cc}
\nabla^2_\theta \mathbf{j}(\theta_c)  + \lambda_c\nabla_\theta^2 \mathbf{g}(\theta_c) & \nabla_{\theta} \mathbf{g}(\theta_c) \cr
\nabla_{\theta} \mathbf{g}(\theta_c)^T & 0\end{array}\right) \left(\begin{array}{c}\dot{\theta}_c\cr \dot{\lambda}_c\end{array}\right)=\left(\begin{array}{c}- \gamma \nabla_{\theta} \mathbf{j}(\theta_c) \cr 1\end{array}\right)
\label{eq:PL_diff_red2_app}
\end{equation}
with retraction factor $\gamma > 0$. A similar idea has been employed by Chen and Jennrich~\cite{ChenJen2002}. Furthermore, to circumvent the potentially time-consuming calculation of the term $\nabla^2_\theta \mathbf{j}(\theta_c)  + \lambda_c\nabla_\theta^2 \mathbf{g}(\theta_c)$ they replace it with a positive definite matrix $\mathbf{w}(\theta_c)$, which depends at most on the first order derivatives of the parameter-to-state map. A possible choice for $\mathbf{w}(\theta_c)$ is the Fisher Information Matrix (FIM), which is for the objective function~\req{eq: reduced objective function} given as 
\begin{equation*}
\eqalign{
\fl \mathbf{w}_{i,j}(\theta_c)&=\mathbf{J}_{\theta_i \theta_j}(\theta_c, S(\theta_c))+\mathbf{J}_{\theta_i u}(\theta_c,S(\theta_c))S_{\theta_j}(\theta_c)+\mathbf{J}_{u \theta_j}(\theta_c,S(\theta_c))S_{\theta_i}(\theta_c) \\
\fl&+\mathbf{J}_{u u}(\theta_c,S(\theta_c))S_{\theta_i}(\theta_c)S_{\theta_j}(\theta_c).}
\end{equation*}
The replacement introduces an approximation error which also results in an underestimation of the profile likelihood. Chen and Jennrich~\cite{ChenJen2002} proved that as the retraction factor $\gamma > 0$ increases, the trajectory of~(\ref{eq:PL_diff_red2_app}) approaches the trajectory of~(\ref{eq:PL_diff_red2}). For $\gamma \rightarrow \infty$, we obtain the singular perturbed system which evolves along the profile \cite{KhalilBook2002}. In this reduced setting, the result from \cite{ChenJen2002} applies directly.

\section{Profile likelihood calculation for the full problem}
\label{sec:Profile likelihood calculation for the full problem}

In the previous section, the reduced problem was considered using the parameter-to-state map $S(\theta)$. The evaluation of $S(\theta)$ requires the accurate numerical simulation of the dynamical system. As this might be computationally inefficient, we introduce optimization and integration based profile likelihood calculation for the non-reduced form of the PDE constrained optimization problem.

%%%%%%
%%%%%%
\subsection{Optimization based profile likelihood calculation}

The optimization based profile likelihood calculation introduced in Section~\ref{sec: Optimization based profile likelihood calculation} relies on the solution of the reduced optimization problem~\req{eq:PL} for every grid point $c_l$. The reduced optimization problem, however, can be replaced by the solution of the PDE constrained optimization problem, 
\begin{equation}
\label{eq:MLE_profile}
\eqalign{
\min_{\theta\in\Theta, u\in W(0,T)}  &\mathbf{J}(\theta,u) \cr
\mbox{s.t. } & u_t+C(\theta,u) = f(\theta)\cr
& u(0)=u_0\cr
& G(\theta,u) = c.}
\end{equation}
We denote the optimal solution by $\left(\theta_c,u_c\right) := \left(\hat\theta(c),\hat u(c)\right)$. This problem can be solved using local optimization, starting at initial points constructed from the previous grid points. For $\theta_c$ and $u_c$, similar extrapolation schemes can be used as for the reduced form.

%%%%%%
%%%%%%
\subsection{Integration based profile likelihood calculation}

For the derivation of the integration based profile likelihood calculation, we consider the Lagrange function of the PDE constrained optimization problem~\req{eq:MLE_profile},
\begin{equation}\label{eq:Lagrange_prof}
\tilde{\mathcal{L}}(\theta,u,p,\lambda)=\mathcal{L}(\theta,u,p) + \lambda(\mathbf{G}(\theta,u)-c)
\end{equation}
with 
\begin{equation*}
%\hspace{-2cm}
\mathcal{L}(\theta,u,p)=\mathbf{J}(\theta,u)+\int^{T}_0{(-\left\langle u,p_t\right\rangle_{V,V^\ast}+\left\langle C(\theta,u)-f(\theta),p\right\rangle_{V^\ast,V})}dt.
\end{equation*}
and Lagrange multipliers $\lambda$ and $p$. The first order optimality conditions for~\req{eq:MLE_profile} at a minimizer $(\theta_c,u_c)$ are
\begin{equation}
\eqalign{
&\nabla_{\theta}\mathcal{L}(\theta_c,u_c,p_c) + \lambda_c \nabla_\theta \mathbf{G}(\theta_c,u_c) = 0\cr
&\nabla_{u}\mathcal{L}(\theta_c,u_c,p_c) + \lambda_c \nabla_{u}\mathbf{G}(\theta_c,u_c) = 0\cr
&\nabla_{p}\mathcal{L}(\theta_c,u_c,p_c) = 0\cr
&\mathbf{G}(\theta_c,u_c) = c.}
\label{eq:1storder_profile}
\end{equation}
The second line is the adjoint equation
\begin{equation}
\left\{
\begin{array}{l}
p_t-C_u(\theta,u)^* p= \mathbf{J}_u(\theta,u)  \\
p(T)=0
\end{array}\right.
\label{eq:adjoint}
\end{equation}
and the third line is the state equation~\req{eq_Model} for $p=p_c$, $u=u_c$ and $\theta=\theta_c$.
%for $p=p_c$, $u=u_c$, $\theta=\theta_c$, i.e., 
%$p_c=-S_{lin}(\theta_c,u_c)^*\mathbf{J}_u(\theta_c,u_c)$
%so that
%\begin{equation}
%\hspace{-2.5cm}\mathbf{J}_u(\theta_c,u_c)S_{\theta_i}(\theta_c)
%=\int_0^T \langle C_{\theta_i}(\theta_c,u_c)-f_{\theta_i}(\theta_c),p_c\rangle_{V^*,V}dt\,, \hspace{1cm} i=1,\ldots,n 
%\label{eq:adjsens}
%\end{equation}
%and the third line in~\req{eq:1storder_profile} is just the state equation \ref{eq:IVP}
%for $u=u_c$, $\theta=\theta_c$, i.e., $u_c=S(\theta_c)$.
Differentiating \req{eq:1storder_profile} with respect to $c$ yields the following system for the evolution of  $(\theta_c,u_c,p_c,\lambda_c)$:
\begin{equation}
\hspace*{-2.5cm}
\underbrace{\left(\begin{array}{cccc}
\nabla_{\theta}^2\mathcal{L}+\lambda_c \nabla^2_\theta \mathbf{G}
&\nabla_{u}\nabla_{\theta}\mathcal{L} + \lambda_c \nabla_{u}\nabla_{\theta} \mathbf{G}
&\nabla_{p}\nabla_{\theta}\mathcal{L}
&\nabla_{\theta} \mathbf{G}
\\
\nabla_{\theta}\nabla_{u}\mathcal{L}  + \lambda_c \nabla_{\theta}\nabla_{u} \mathbf{G}
&\nabla_{u}^2\mathcal{L}  + \lambda_c \nabla_{u}^2 \mathbf{G}
&\nabla_{p}\nabla_{u}\mathcal{L}
&\nabla_{u} \mathbf{G}
\\
\nabla_{\theta}\nabla_{p}\mathcal{L}
&\nabla_{u}\nabla_{p}\mathcal{L}
&0
&0\\
\nabla_{\theta} \mathbf{G}^T
& \nabla_{u} \mathbf{G}^T
&0
&0
\end{array}\right)}_{M_\mathrm{full}(\theta_c,u_c,p_c,\lambda_c)}
\left(\begin{array}{c}
\dot{\theta}_c\\
\dot{u}_c\\
\dot{p}_c\\
\dot{\lambda}_c
\end{array}\right)
=\left(\begin{array}{c}
0\\
0\\
0\\
1
\end{array}\right).
\label{eq:PL_diff_unred}
\end{equation}
where we skipped the arguments $(\theta_c,u_c,p_c)$ of the derivatives of $\mathcal{L}$ and $\mathbf{G}$. The derivatives of $\mathcal{L}$ are
\begin{equation}\label{eq:second_derivatives}
\eqalign{
\fl \mathcal{L}_{\theta_i \theta_j}(\theta_c,u_c,p_c)&=\mathbf{J}_{\theta_i \theta_j}(\theta_c,u_c)+\int^{T}_0{\left\langle C_{\theta_i \theta_j}(\theta_c,u_c)-f_{\theta_i \theta_j}(\theta_c),p_c\right\rangle_{V^\ast,V}}dt \\
\fl \nabla_{u}\mathcal{L}_{\theta_i}(\theta_c,u_c,p_c) &=\mathbf{J}_{\theta_i u}(\theta_c, u_c)+C_{\theta_i u}(\theta_c,u_c)^\ast p_c\\
\fl \nabla_{p}\mathcal{L}_{\theta_i}(\theta_c,u_c,p_c) &=C_{\theta_i}(\theta_c,u_c)-f_{\theta_i}(\theta_c)\\
\fl \nabla^2_{u}\mathcal{L}(\theta_c,u_c,p_c)&=\mathbf{J}_{uu}(\theta_c,u_c)+C_{uu}(\theta_c,u_c)^\ast p_c \\
\fl \nabla_{u}\nabla_{p}\mathcal{L}(\theta_c,u_c,p_c)&=\partial_t+C_u(\theta_c,u_c)
}
\end{equation}
The other mixed partial derivatives of the Langrange function follow by symmetry if all involved functions are twice continuously differentiable.

The trajectories of~\req{eq:PL_diff_unred} provide the profile likelihood for the non-reduced problem, namely $\theta_c$ and $u_c$. For the numerical integration an explicit or implicit time stepping scheme can be used. Similarly to the reduced problem, the approximation errors can be reduced by introduction of a retraction term,
\begin{equation}
\hspace*{-2.5cm}
\left(\begin{array}{cccc}
W_{uu\,c}
&W_{u\theta\,c}
&W_{p\theta\,c}
&\nabla_{\theta}\mathbf{G}
\\
W_{\theta u\,c}
&W_{uu\,c}
&W_{pu\,c}
&\nabla_{u}\mathbf{G}
\\
W_{\theta p\,c}
&W_{up\,c}
&W_{pp\,c}
&0\\
\nabla_{\theta}\mathbf{G}^T
&\nabla_{u}\mathbf{G}^T
&0
&0
\end{array}\right)
\left(\begin{array}{c}
\dot{\hat{\theta}}_c\\
\dot{\hat{u}}_c\\
\dot{\hat{p}}_c\\
\dot{\hat{\lambda}}_c
\end{array}\right)
=\left(\begin{array}{c}
-\gamma \nabla_\theta \mathcal{L}\\
-\gamma \nabla_u \mathcal{L}\\
-\gamma \nabla_p \mathcal{L}\\
1
\end{array}\right),
\label{eq:PL_diff_unred_retr}
\end{equation}
with retraction factor $\gamma > 0$. The retraction damps the accumulation of numerical errors and ensures a more accurate profile likelihood approximation. Furthermore, the retraction allows for the replacement of the matrix $M_\mathrm{full}(\theta_c,u_c,p_c,\lambda_c)$ with a positive definite matrix to circumvent the need for second order information. The resulting approximation error can be controlled using $\gamma$. A large retraction factor results, however, in an increased stiffness of the dynamical system. Extending the proof in~\cite{ChenJen2002} from finite dimensions and ODEs to the PDE setting in function spaces we can show that the difference between solutions to the original system and the approximated one with retraction can be made arbitrarily small by an appropriate choice of the retraction factor $\lambda_c$, see 
Proposition \ref{prop:retr_full} below.
For this purpose we abbreviate $\xi_c=(\theta_c,u_c,p_c)$, $\hat{\xi}_c=(\hat{\theta}_c,\hat{u}_c,\hat{p}_c)$, $\mathbf{G}(\xi)=\mathbf{G}(\theta,u)$, $X=\R^{n_\theta}\times W(0,T)^2$, so that we can rewrite \req{eq:PL_diff_unred} and \req{eq:PL_diff_unred_retr} more compactly as
\begin{equation}
\hspace*{-2.5cm}
\left(\begin{array}{cc}
\nabla_{\xi}^2 \mathcal{L}(\xi_c)+\lambda_c \nabla_{\xi}^2 \mathbf{G}(\xi_c)&\nabla_{\xi}\mathbf{G}(\xi_c)
\\
\nabla_{\xi}\mathbf{G}(\xi_c)^T&0
\end{array}\right)
\left(\begin{array}{c}
\dot{\xi}_c\\
\dot{\lambda}_c
\end{array}\right)
=\left(\begin{array}{c}
0\\
1
\end{array}\right),
\label{eq:PL_diff_unred_x}
\end{equation}

\begin{equation}
\hspace*{-2.5cm}
\left(\begin{array}{cc}
W_c&\nabla_{\xi}\mathbf{G}(\hat{\xi}_c)
\\
\nabla_{\xi}\mathbf{G}(\hat{\xi}_c)^T&0
\end{array}\right)
\left(\begin{array}{c}
\dot{\hat{\xi}}_c\\
\dot{\hat{\lambda}}_c
\end{array}\right)
=\left(\begin{array}{c}
-\gamma \nabla_{\xi} \mathcal{L}(\hat{\xi}_c)\\
1
\end{array}\right).
\label{eq:PL_diff_unred_retr_x}
\end{equation}
We assume that the family of linear operators $W_c:X\to X^*$ satisfies the following properties:
\begin{equation}\label{Wsymm}
\hspace*{-2.5cm}
\forall c\in[c_0,c_1] \ \forall \xi,\zeta \in X:  \langle W_c \xi, \zeta\rangle_{X^*,X} =\langle W_c \zeta, \xi\rangle_{X^*,X}
\quad \mbox{(symmetry)}
\end{equation}
\begin{equation}\label{Wpos}
\hspace*{-2.5cm}
\forall c\in[c_0,c_1] \ \forall \xi \in X:  \langle W_c \xi, \xi\rangle_{X^*,X} \geq \gamma_W \|\xi\|_X^2
\quad \mbox{(positivity)}
\end{equation}
\begin{equation}\label{Wbd}
\hspace*{-2.5cm}
\exists \bar{M}>0\ \forall c\in[c_0,c_1]: \ \| W_c\|=\sup_{\xi,\zeta\in X,\, \|\xi\|_X\leq 1, \, \|\zeta\|_X\leq1} \hspace*{-1cm}\langle W_c \xi, \zeta\rangle_{X^*,X} \leq \bar{M}_W
\quad \mbox{(boundedness)
}
\end{equation}
as well as the following sufficient second order condition at the minimizers $(\xi_c,\lambda_c)$
\begin{equation}\label{ssc}
\hspace*{-2.5cm}
\exists \gamma_L>0 \, \forall c\in[c_0,c_1] \, \forall \zeta\in \nabla_{\xi} \mathbf{G}(\xi_c)_\bot: \, 
\langle (\nabla_{\xi}^2 \mathcal{L}(\xi_c)+\lambda_c \nabla_{\xi}^2 \mathbf{G}(\xi_c)) \zeta, \zeta\rangle_{X^*,X} \geq \gamma_L \|\zeta\|_X^2
\end{equation}
where $\nabla_{\xi} \mathbf{G}(\xi_c)_\bot=\{\zeta\in X \, : \, \langle\nabla_{\xi} \mathbf{G}(\xi_c),\zeta\rangle_{X^*,X}=0\}$ is the tangential cone corres\-pon\-ding to the equality constraint $\mathbf{G}(\xi)=0$. 
\begin{proposition}\label{prop:retr_full}
Let $\mathbf{G}$, $\mathcal{L}$ be twice continuously differentiable, and let \req{Wsymm}, \req{Wpos}, \req{ssc} be satisfied and let $\xi_c$, $\hat{\xi}_c$, $c\in[c_0,c_1]$ be solutions to \req{eq:PL_diff_unred_x} and \req{eq:PL_diff_unred_retr_x}, respectively.\\
Then for any $\kappa>0$, and for any $\tilde{\epsilon}>0$ sufficiently small, there exists $\rho>0$ sufficiently small and $\lambda>0 $ sufficiently large, such that if $e_{c_0}<\rho$ then 
\begin{equation}\label{bdexpdecay}
\forall c\in[c_0,c_1] \  \|\hat{\xi}_{c}-\xi_{c}\|_X\leq\rho \mbox{ and } e_c\leq  \frac{\tilde{\epsilon}}{\kappa} + e_{c_0} \exp(-\kappa(c-c_0))
\end{equation}
holds, where $e_c= \langle W_c \hat{\xi}_{c}-\xi_{c}, \hat{\xi}_{c}-\xi_{c}\rangle_{X^*,X}\geq \gamma_W \|\hat{\xi}_{c}-\xi_{c}\|_X^2$.\\
Moreover, for any $\varepsilon>0$ and any $\tilde{c}\in(c_0,c_1]$ there exists $\lambda>0$ such that 
\begin{equation}\label{conv}
\forall c\in[\tilde{c},c_1] \  \|\hat{\xi}_{c}-\xi_{c}\|_X\leq\varepsilon\,.
\end{equation}
\end{proposition}
The proof (see the Appendix) shows that $\lambda=\lambda_c$ can be chosen adaptively, depending on the artificial time parameter $c$.

%%%%%%
%%%%%%
%%%%%%
\section{Comparison of full and reduced formulation of integration based profile likelihood calculation}
\label{sec:Comparison of full and reduced formulation of integration based profile likelihood calculation}

The full and reduced formulations of integration based profile calculation provide different view points on the problem. In the following, we will establish equivalence under the assumption of the identity $u_c = S(\theta_c)$. In addition, the computational implementation will be discussed.

%%%%%%
%%%%%%
\subsection{Equivalence of calculated profile likelihoods}

The reduced formulation~\req{eq:PL_diff_red2} and the full formulation~\req{eq:PL_diff_unred},
\begin{equation*}
\hspace{-2cm}
 M_\mathrm{red}(\theta_c,\lambda_c)\left(\begin{array}{c}\dot{\theta}_c \\ \dot{\lambda}_c\end{array}\right)=
\left(\begin{array}{c}0\\ 1\end{array}\right)
\quad \mathrm{and} \quad 
M_\mathrm{full}(\theta_c,u_c,p_c,\lambda_c)\left(\begin{array}{c}\dot{\theta}_c\\ \dot{u}_c\\ \dot{p}_c \\ \dot{\lambda}_c\end{array}\right)= 
\left(\begin{array}{c}0\\0\\0\\1\end{array}\right),
\end{equation*}
provide two alternative approaches to calculate the profile likelihood path $\theta_c$. The validity of the state equation for $u_c$, which is ensured by the initial condition satisfying~\req{eq:1storder_profile} gives the identity $u_c=S(\theta_c)$. With this identity as well as the evolution \req{eq:PL_diff_unred}, we will show the equivalence of both approaches in the following.

\begin{proposition}
Under Assumptions~\ref{Ass:weak_solution} and~\ref{Ass:diff} solving the full system~\req{eq:PL_diff_unred} and the reduced system~\req{eq:PL_diff_red2} yields the same profile likelihood path $\theta_c$.
\end{proposition}

\begin{proof}
The operators $\nabla_u \nabla_p \mathcal{L}(\theta_c,u_c,p_c,\lambda_c)$ and $\nabla_p\nabla_u\mathcal{L}(\theta_c,u_c,p_c,\lambda_c)$ represent the linearised state and the adjoint equation, respectively and are thus invertible under Assumptions \ref{Ass:weak_solution} and \ref{Ass:diff}.
Therefore we can formally eliminate the variables $(\dot{u}_c,\dot{p}_c)$ by means of the second and third line in the system \req{eq:PL_diff_unred}, which yields
\begin{equation}
\eqalign{
 \dot{u}_c=-(\nabla_u\nabla_p\mathcal{L})^{-1}
\nabla_\theta\nabla_p\mathcal{L}\,\dot{\theta}_c\\
 \dot{p}_c=(\nabla_p\nabla_u\mathcal{L})^{-1}
\Bigl(\nabla_u^2\mathcal{L}(\nabla_u\nabla_p\mathcal{L})^{-1}
\nabla_\theta\nabla_p\mathcal{L}-\nabla_\theta\nabla_u\mathcal{L}
\Bigr)\dot{\theta}_c\,,
}
\label{eq:ucpc}
\end{equation}
where we have skipped the arguments $(\theta_c,u_c,p_c)$ of the Lagrangian for better readability.
Inserting this into \req{eq:PL_diff_unred} yields
\begin{equation*}
\left(\begin{array}{cc}
\tilde{M}+\lambda_c \nabla_{\theta}^2 g(\theta_c)& \nabla_{\theta} g(\theta_c)\cr
\nabla_{\theta} g(\theta_c)^T & 0\end{array}\right)
\left(\begin{array}{c}\dot{\theta}_c\cr \dot{\lambda}_c\end{array}\right)=\left(\begin{array}{c}0\cr 1\end{array}\right)
\end{equation*}
with 
\begin{equation*}
\hspace{-2cm}
\eqalign{
\tilde{M}&=\nabla_\theta^2\mathcal{L}
+\nabla_u\nabla_\theta\mathcal{L}(-\nabla_u\nabla_p\mathcal{L})^{-1}
\nabla_\theta\nabla_p\mathcal{L}
-\nabla_p\nabla_\theta\mathcal{L}(\nabla_p\nabla_u\mathcal{L})^{-1}
\nabla_\theta\nabla_u\mathcal{L}\\
&-\nabla_p\nabla_\theta\mathcal{L}(\nabla_p\nabla_u\mathcal{L})^{-1}
\nabla_u^2\mathcal{L}(-\nabla_u\nabla_p\mathcal{L})^{-1}
\nabla_\theta\nabla_p\mathcal{L}\,.
}
\end{equation*}
Thus to show equivalence with \req{eq:PL_diff_red2} it only remains to verify that
$\tilde{M}=\nabla_{\theta}^2 \mathbf{j}(\theta_c)$.
With the second derivatives according to~\req{eq:second_derivatives} and
$(-\nabla_p \mathcal{L}_{\theta_i})(\nabla_p\nabla_u\mathcal{L})^{-1}=S_{\theta_i}(\theta_c)$
we get
\begin{equation*}
\eqalign{
\fl\tilde{M}_{i,j}&=\mathbf{J}_{\theta_i \theta_j}(\theta_c,S(\theta_c))+\int^T_0{\left\langle{C_{\theta_i \theta_j}(\theta_c,S(\theta_c))-f_{\theta_i \theta_j}(\theta_c),p_c}\right\rangle_{V^\ast,V}}dt+\mathbf{J}_{\theta_i u}(\theta_c,S(\theta_c))S_{\theta_j}(\theta_c)\\
\fl&+\int^T_0\left\langle{C_{\theta_i u}(\theta_c,S(\theta_c))S_{\theta_j}(\theta_c),p_c}\right\rangle_{V^\ast,V}dt
+\int^T_0{\left\langle C_{uu}(\theta_c,S(\theta_c))S_{\theta_i}(\theta_c)S_{\theta_j}(\theta_c),p_c\right\rangle_{V^\ast,V}}dt\\
\fl &+\mathbf{J}_{uu}(\theta_c,S(\theta_c))S_{\theta_i}(\theta_c)S_{\theta_j}(\theta_c)
+\mathbf{J}_{u \theta_j}(\theta_c, S(\theta_c))S_{\theta_i}(\theta_c) \\
\fl &+\int^T_0{\left\langle{C_{u\theta_j}(\theta_c,S(\theta_c))S_{\theta_i}(\theta_c),p_c}\right\rangle_{V^\ast,V}}dt\,.
}
\end{equation*}
Using the definition of $h_{ij}^{\theta_c}$ and the identity
\begin{equation*}
\int^T_0{\left\langle{h_{ij}^{\theta_c},p_c}\right\rangle_{V^\ast,V}}dt=\mathbf{J}_{u}(\theta_c,S(\theta_c))S_{\theta_i \theta_j}(\theta_c)
\end{equation*}
we obtain
\begin{equation*}
\hspace{-2.5cm}
\eqalign{
\tilde{M}_{i,j}&=\mathbf{J}_{\theta_i \theta_j}(\theta_c,S(\theta_c))+\mathbf{J}_{\theta_i u}(\theta_c,S(\theta_c))S_{\theta_j}(\theta_c)+\mathbf{J}_{uu}(\theta_c,S(\theta_c))S_{\theta_i}(\theta_c)S_{\theta_j}(\theta_c)\\
\fl&+\mathbf{J}_{u \theta_j}(\theta_c, S(\theta_c))S_{\theta_i}(\theta_c)+\mathbf{J}_u(\theta_c,S(\theta_c))S_{\theta_i \theta_j}(\theta_c),
}
\end{equation*}
which is the same as~\req{second_deriv_reduced} and therefore establishes equivalence.
\end{proof}

From the equivalence of~\req{eq:PL_diff_red2} and~\req{eq:PL_diff_unred}, we conclude that also the stabilized versions~\req{eq: reduced integration-based profile calculation, stabilized} and~\req{eq:PL_diff_unred_retr} yield the same results in the absence of numerical integration errors.

%%%%%%
%%%%%%
\subsection{Implementation and computational properties}

In the previous section we established equivalence of $\theta_c$ for~\req{eq:PL_diff_red2} and~\req{eq:PL_diff_unred}. This equivalence is however not ensured for the result of the trajectories of~\req{eq:PL_diff_red2} and~\req{eq:PL_diff_unred} (or their stabilized versions~\req{eq: reduced integration-based profile calculation, stabilized} and~\req{eq:PL_diff_unred_retr}) obtained by numerical simulation. The implementation and computational requirements of full and reduced systems are considerably different.

For the numerical simulation of the reduced system, the matrix-vector product $M_\mathrm{red}(\theta_c) (\dot\theta_c,\dot\lambda_c)^T$ has to be evaluated. This requires the numerical simulation of the model~\req{eq_Model} for every point $(\theta_c,\lambda_c)^T$ and either $n (n+2)/2$ linear forward PDE solves or one linear backward PDE solve (see Section~\ref{sec: Integration based profile likelihood calculation - reduced}). This implicit numerical simulation can exploit sophistical numerical solvers, requires minimal storage but can be computationally demanding. In contrast, the full system provides an explicit form. When applying an iterative solver, e.g. a CG method, only matrix vector products are needed. Applying $M_\mathrm{full}(\theta_c,u_c,p_c,\lambda_c)$ to a vector $(\dot\theta_c,\dot u_c, \dot p_c, \dot\lambda_c)^T$ only requires the evaluation of the linearization of the differential operator $\partial_t+C(\theta_c,.)$ and its adjoint, so no PDE solution. The discretization of $u_c$ and $p_c$ in space and time can however require significant storage. 

In this paper the stabilized reduced system~\req{eq: reduced integration-based profile calculation, stabilized} is implemented. For the numerical simulation an adaptive solver is employed.

%%%%%%
%%%%%%
%%%%%%
\section{Numerical evaluation of integration based profile likelihood calculation}
\label{sec:Numerical evaluation of integration based profile likelihood calculation}

In the following, we will illustrate the properties of the proposed integration based profile likelihood calculation method. For this purpose, we study a biological application, i.e. the PDE model for gradient formation in a yeast cell. We consider a realistic measurement set-up but use artificial experimental data. This enables the comparison of the methods with the ground-truth available.

%%%%%%
%%%%%%
\subsection{Mathematical model for gradient formation in fission yeast}

To assess the properties of the proposed approach, we consider a model for gradient formation in fission yeast. Fission yeast cells are rod-shaped and their division is controlled by a gradient of the protein Pom1p in the cell membrane. Hersch et al.~\cite{Hersch2015} modelled the dynamics of the concentration of Pom1p at a position $x$, $u(t,x)$ with units $\#/\mu m$, by
\begin{equation}
%\left\{
\begin{array}{ll}
u_{t} = Du_{xx} -\alpha u^2 + \frac{\beta}{\sqrt{2\pi}\rho}e^{-x^2/2\rho^2} &\mathrm{for} \  ]0,T[\times ]-L,L[\\
\frac{\partial u}{\partial\nu} = 0 &\mathrm{for} \ ]0,T[\times\lbrace{-L,L\rbrace}\\
u = 0 &\mathrm{for} \lbrace{t=0\rbrace\times ]-L,L[}
\end{array}
%\right.
\label{eq:App_PDE}
\end{equation}
with diffusion coefficient $D$, dimerisation rate $\alpha$, influx rate $\beta$ and source width $\rho$. The length along the membrane from the tip of the cell to the center is denoted by $L$. Model~\req{eq:App_PDE} meets Assumption~\ref{Ass:weak_solution}.

We implemented the method of lines for model~\req{eq:App_PDE} in MATLAB. The system of ODEs was implemented in AMICI (\textit{https://github.com/ICB-DCM/AMICI}) \cite{KazeroonianFro2016}. AMICI generates the first and second order sensitivity equations, enabling the evaluation of the Hessian and the Fisher Information Matrix. For the simulation, AMICI exploits the SUNDIALS solver suite~\cite{HindmarshBro2005}.

%%%%%%
%%%%%%
\subsection{Artificial experimental data}
For a realistic evaluation of different profile likelihood calculation methods, we generated artificial experimental data similar to previously published datasets for the considered process (see~\cite{Saunders2012}). Our artificial dataset consists of three individual datasets: 
\begin{itemize}
\item \textit{Concentration profile:}
The concentration profile provides the relative abundance of the signalling molecule along the membrane at time $t = 100 \,  s$. The interval $]-L,L[$ is divided in 60 equally sized regions $\Omega_k$, yielding the observation operators
\[Q_{k}(\theta,u) = s_1 \int_{\Omega_k} u(t = 100,x) dx \quad \mathrm{for}Ê\quad k = 1,\ldots,60,\] 
with scaling factor $s_1$ and region $\Omega_k = \left[-7 + \frac{7}{30} (k-1),-7 + \frac{7}{30} k\right]$, $k = 1,\ldots,60$.
\item \textit{Time course:}
The time course data provide the scaled overall protein abundance at 10 equally spaced time points $t_k \in [0,60] \, s$. The observation operators are 
\[Q_{60+k}(\theta,u) = s_2 \int_{-L}^{L} u(t_k,x) dx \quad \mathrm{for}Ê\quad k = 1,\ldots,10,\]
with scaling factor $s_2$.
\item \textit{Quantification:}
The quantification provides the absolute abundance of the signalling molecule at time point $t = 100 \, s$,  
\[Q_{71}(\theta,u) = \int_{-L}^{L} u(t = 100,x) dx.\] 
\end{itemize}
The artificial data sets were obtained by simulating model~\req{eq:App_PDE} for the parameters provided in Table~\ref{tab:App_Data}, and subsequently adding normally distributed measurement noise. The final data set is the mean $\bar{y}_k$ and standard deviation $\sigma_k$, $k = 1, \ldots, 71$, as depicted in Figure~\ref{fig:App_Data}. In biological applications the acquisition of measurement data is often challenging. Instead of a single highly-informative experiment, merely a series of measurements with low information content can be performed. This commonly results in observation operators $Q(\theta,u)$ with a non-standard structure.

\begin{table}
\begin{tabular}{||p{2cm}||p{2cm}||p{2cm}||p{2.5cm}|p{2.5cm}||c||} 
\hline\hline
\centering parameter & \centering true & \centering estimated &  \multicolumn{2}{c||}{ 95\% confidence interval} &units\\ 
\centering names & \centering value $\theta$ & \centering value $\widehat{\theta}$ & \centering lower bound & \centering upper bound &\\ 
\hline\hline
%\centering parameter names & \centering $\quad$ true $\quad$ value $\theta$ & \centering estimated value $\widehat{\theta}$ &  \multicolumn{2}{c|}{ $\quad$ $\quad$ 95\% confidence $\quad$ $\quad$ interval, $\mathrm{CI}_{0.95,\theta_i}$} &units\\ 
%\hline
\centering $D$ &\centering $0.10$& \centering $ 0.15$ & \centering $0$ & \centering$> 12$&$\mu m^2/s$\\ 
\centering $\alpha$ &\centering $4.00\times10^{-4}$& \centering $ 6.07\times10^{-4}$ & \centering $ 4.60\times10^{-5}$ &\centering $>1 $&$\mu m^3/(\#\cdot s)$\\
\centering $\beta$ &\centering $8.00\times10^{3}$& \centering $ 1.21\times10^{4}$ & \centering $9.99\times10^{2}$ &\centering $> 5\times 10^{8}$&$\# \mu m / s$\\ 
\centering $\rho$ &\centering $0.60$& \centering $ 0.60$ & \centering $0.38$ & \centering$0.77$&$\mu m$\\ 
\centering $s_1$ &\centering $2.87\times10^{-4}$& \centering $ 2.95\times10^{-4}$ & \centering $2.45\times10^{-4}$ &\centering $3.47\times10^{-4}$&$ui / \#$\\
\centering $s_2$ &\centering $2.70\times10^{-5}$& \centering $ 2.72\times10^{-5}$ & \centering $2.26\times10^{-5}$ &\centering $3.34\times10^{-5}$&$ui / \#$\\
\hline\hline
\end{tabular}
\caption{\textbf{Parameter values and confidence intervals.} The parameter estimates are obtained using multi-start local optimization. The confidence intervals to a confidence level of $95\%$ are computed with the integration based profile likelihood calculation in ODE formulation with the Hessian. For confidence intervals which extend beyond the considered parameter region we write $>$ or $<$ bound, indicating practical non-identifiability. All quantities are provided in seconds $s$, micrometer $\mu m$, number of molecules $\#$ and unit of fluorescence intensity $ui$.}
\label{tab:App_Data}
\end{table}

\begin{figure}
\includegraphics[width=\textwidth]{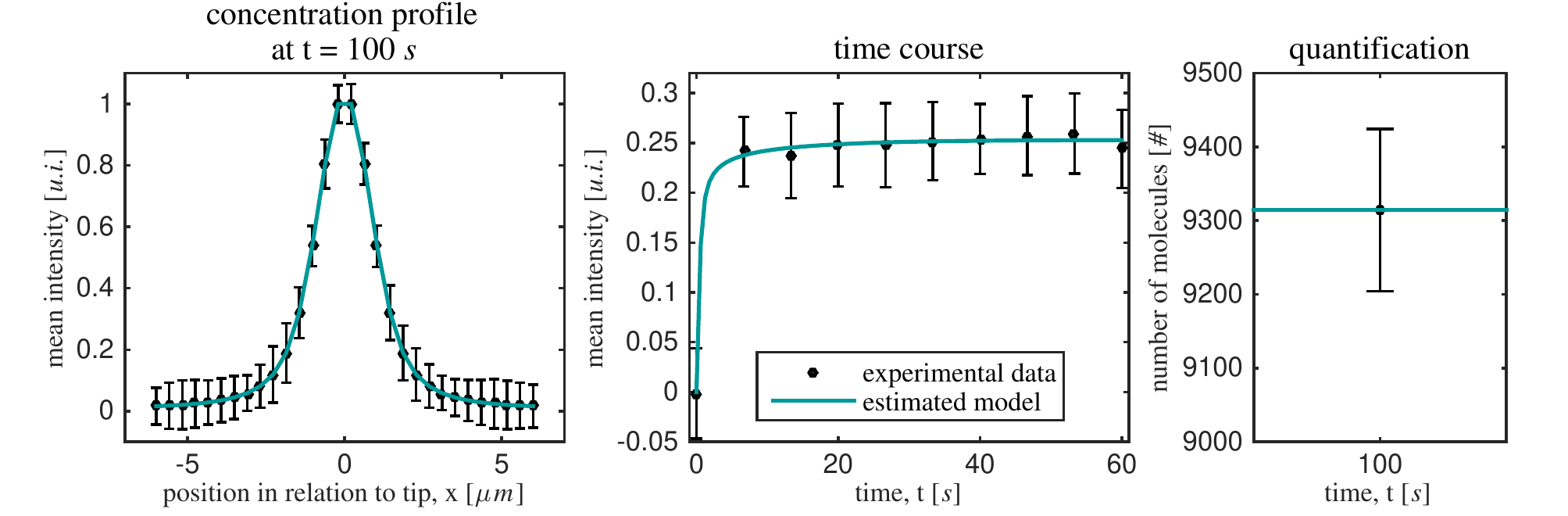}
\caption{\textbf{Artificial experimental data and model fit.} The measurement data ($\bullet$) and its standard deviation (error bar) are depicted along with the best model fit~(\textcolor{SimBlue}{\rule[.6mm]{0.25cm}{2pt}}) for (left) the concentration profile, (middle) the time course and (right) the protein abundance.}
\label{fig:App_Data}
\end{figure}

%%%%%%
%%%%%%
\subsection{Parameter optimization}

For the estimation of the unknown parameter vector $\theta = (D,\alpha,\beta,\rho,s_1,s_2)^T$ we use maximum likelihood estimation. The measurement noise is assumed to be normally distributed with the measured standard deviation $\sigma_k$ (see error bars in Figure~\ref{fig:App_Data}). This yields the reduced optimization problem 
\begin{equation}
\min_{\theta\in\R^{n}}  \mathbf{j}(\theta) = \frac{1}{2} \sum^{71}_{k=1} \left(\log\left(2\pi \sigma^2_{k}\right) + \left(\frac{\bar{y}_{k}-Q_{k}(\theta,S(\theta))}{\sigma_{k}}\right)^2\right)
\label{eq:App_OptProb}
\end{equation} 
in which $S(\theta)$ denotes the parameter-to-state map of~\req{eq:App_PDE}, which is evaluated using numerical integration. The optimum of~\req{eq:App_OptProb} is determined using multi-start local optimization. Therefore, the MATLAB optimizer \textit{fmincon} is initialised at 100 different starting points chosen with a space filling design, i.e. latin hyper cube sampling. More than 90\% of these optimizations converged to the same optimal likelihood value, which we assume to be global. The estimation results are shown in Table~\ref{tab:App_Data} and Figure~\ref{fig:App_Data}.

%%%%%%
%%%%%%
\subsection{Profile likelihood calculation}

To assess the uncertainty of the parameter estimate, we compute the profile likelihoods $\mathrm{PL}_{\theta_i}(c)$, $i = 1,\ldots,6$. Therefore we employ existing
\begin{itemize}
\item optimization based profile likelihood calculation with 0th and 1st-order proposal, 
\end{itemize}
as well as the proposed
\begin{itemize}
\item integration based profile likelihood calculation with Hessian and
\item integration based profile likelihood calculation with FIM.
\end{itemize}
For the integration based methods we compare the numerical implementation as DAE system~\req{eq:PL_diff_red2_app} %,
%$$M_\mathrm{red}(\theta_c,\lambda_c) \left(\begin{array}{c} \dot\theta_c \\ \dot\lambda_c \end{array}\right) = \left(\begin{array}{c} \nabla_\theta \mathbf{j}(\theta_c) \\ 1 \end{array}\right),$$
and as ODE system,
$$\left(\begin{array}{c} \dot\theta_c \\ \dot\lambda_c \end{array}\right) = M_\mathrm{red}(\theta_c,\lambda_c)^+\left(\begin{array}{c} \nabla_\theta \mathbf{j}(\theta_c) \\ 1 \end{array}\right),$$
with $M_\mathrm{red}(\theta_c,\lambda_c)^+$ denoting the Moore-Penrose pseudo-inverse of $M_\mathrm{red}(\theta_c,\lambda_c)$. The ODE implementation has been suggested by Chen and Jennrich~\cite{ChenJen2002}. All methods were implemented in the Parameter EStimation TOolbox (PESTO) for MATLAB (\textit{https://github.com/ICB-DCM/PESTO}). The optimization based calculation exploits the MATLAB function \textit{fmincon} with a user supplied gradient of the objective function. The integration based calculation is implemented using the MATLAB function \textit{ode15s}, which is suited for ODEs and DAEs. This implementation is also included in PESTO.
%For the calculation of the integration based profile calculation we consider the reduced system (\ref{eq:PL_diff_red2}) with the Hessian matrix of the objective function obtained from the second order sensitivity equations of the model, which results in a differential algebraic equation (DAE). By inversion of $M_{red}$ we obtained an ordinary differential equation (ODE). For the approximative profile likelihood calculation (\ref{eq:PL_diff_red2_app}) we used the FIM of the objective function based on the first order sensitivities of the model. Here we also consider the DAE and ODE case both with an correction parameter $\gamma$, which was chosen manually. All DAEs and ODEs are solved with the MATLAB function \textit{ode15s}. The implementation is part of the freely available PESTO toolbox (\sh{@ Jan: gibt es schon ein github f\"{u}r PESTO? Sollten wir f\"{u}r das Paper erstellen}).\\

The profile likelihood calculation using the aforementioned methods provided consistent results. Optimization based profile likelihood calculation, integration based profile likelihood calculation using the Hessian and integration based profile likelihood calculation using the FIM (with retraction factor $\gamma > 2$) indicate that for the given data set $\rho$, $s_1$ and $s_2$ are practically identifiable for a confidence level of 95\% while $D$, $\alpha$ and $\beta$ are practically non-identifiable (see Table~\ref{fig:App_Data} and Figure~\ref{fig:App_Prof}A). For $\gamma < 6$, integration based profile likelihood calculation using the FIM yielded an underestimation of the profile likelihood. This effect is worse for practically non-identifiable parameters than for practically identifiable parameters. However, for increasing values of $\gamma$ the profile likelihood converged to the profile likelihood obtained with the optimization based method (Figure~\ref{fig:App_Prof}A (right)). Integration based profile likelihood calculation using the Hessian provided accurate results independent of $\gamma$. Differences in the path $\theta_c$ resulting from the implementation of DAE or ODE were negligible.

% For the parameters $\rho$, $s_1$ and $s_2$ all methods yield the same profile likelihood (see Figure~\ref{fig:App_Prof}~A). The correction parameter $\gamma$ had little influence on the results. For the parameters $D$, $\alpha$ and $J$ the results are merely consist for sufficient large retraction factors $\gamma$.  had to be chosen large enough for the approximative  formulation. With increasing values of $\gamma$ the slope of the profile converged to the profile likelihood obtained with the optimization based method and we found that for the considered problem $\gamma = 10$ is sufficient. Raising $\gamma$, however, results in an increased stiffness of the DAE or ODE and consequently higher computation times.

The comparison of the computation time for the different methods revealed that integration based profile likelihood calculation using the Hessian was computationally more efficient than the other methods. For the practically identifiable parameter $\rho$ the implementation as ODE results in a speed-up by a factor of five compared to optimization based profile calculation with 0th order proposal (Figure~\ref{fig:App_Prof}A (left)). For the practically non-identifiable parameter $\alpha$ the efficiency improvement was a factor 144 compared to the optimization based methods with 0th order proposal and a factor five compared to optimization based methods with 1st order proposal. One reason for the improved computational efficiency was the adaptive choice of the evaluation points, which allowed for larger steps in regions with smaller curvature. This effect was reduced for the FIM, as the retraction increases the stiffness of the DAE or ODE. The decrease in the step sizes due to the stiffness yielded more function evaluations and an increased computation time. Surprisingly, this increase also outweighed the higher computation cost of computing second order sensitivities. Furthermore our analysis of the computation times demonstrated that for this problem the ODE implementation was computationally more efficient than the DAE implementation.

\begin{figure}
\centering
\includegraphics[width=\textwidth]{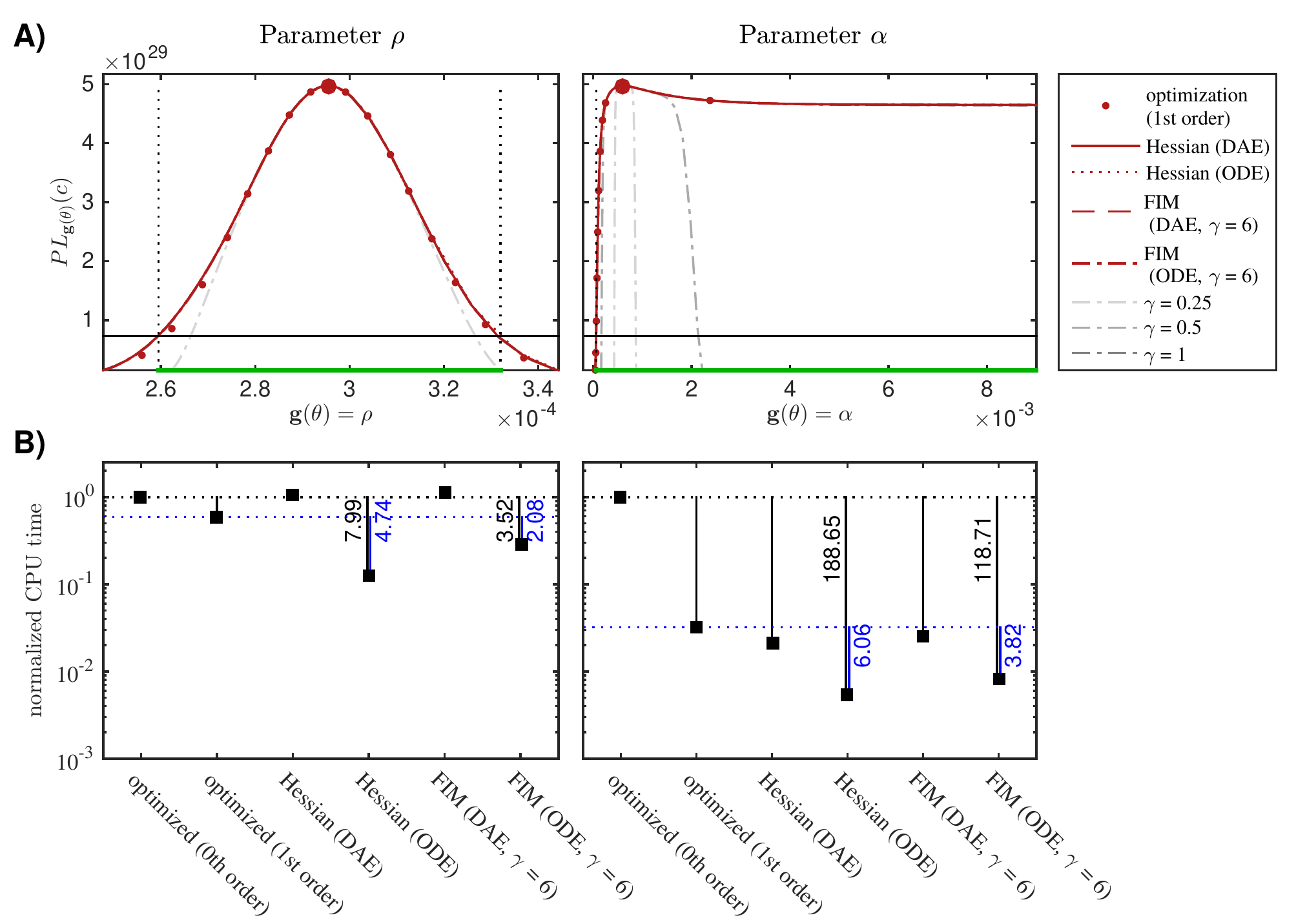}
\caption{\textbf{Profile likelihood calculation for individual parameters $\theta_i$.}
(A) Profile likelihood for $\rho$ and $\alpha$ calculated using the optimization based and integration based methods. For the practically identifiable parameter $\rho$ (left), almost all profiles agree perfectly. For the practically non-identifiable parameter $\alpha$ (right), integration based methods using the Fisher Information Matrix and $\gamma < 2$ underestimate the profile likelihood.
(B) Normalised CPU time for different profile likelihood calculation methods. The numbers indicate the speed up compared to the optimization based method with 0th order proposal (\textcolor{black}{\rule[.6mm]{0.25cm}{1pt}}) or the 1st order proposal (\textcolor{blue}{\rule[.6mm]{0.25cm}{1pt}}).}
\label{fig:App_Prof}
\end{figure}

Integration based profile calculation methods allow for the analysis of individual parameters $\mathbf{g}(\theta) = \theta_i$ but also for more complex expressions. We considered the parameter ratio $\mathbf{g}(\theta,u) = \frac{\alpha}{\beta}$. While the individual parameters possess an unbounded confidence interval and are practically non-identifiable, the ratio is practically identifiable and possesses a finite confidence interval (Figure~\ref{fig:App_ProfComb}A). This indicates that influx (related to $\beta$) and outflow (related to $\alpha$) are balanced. In addition, the analysis of the quotient of the molecule abundance at the tip compared to the abundance at $x=2$ for the time point $t = 100$, $\mathbf{G}(\theta,u) = u(t=100,x=2)/u(t=100,x=0)$, revealed that the steepness of teh gradient is well determined (Figure~\ref{fig:App_ProfComb}B).

In summary, the numerical evaluation for Pom1p signalling revealed the accuracy and efficiency of the integration based profile likelihood calculation methods. Beyond individual parameters, integration based methods facilitated uncertainty analysis for a range of scalar model properties.

\begin{figure}
\includegraphics[width=\textwidth]{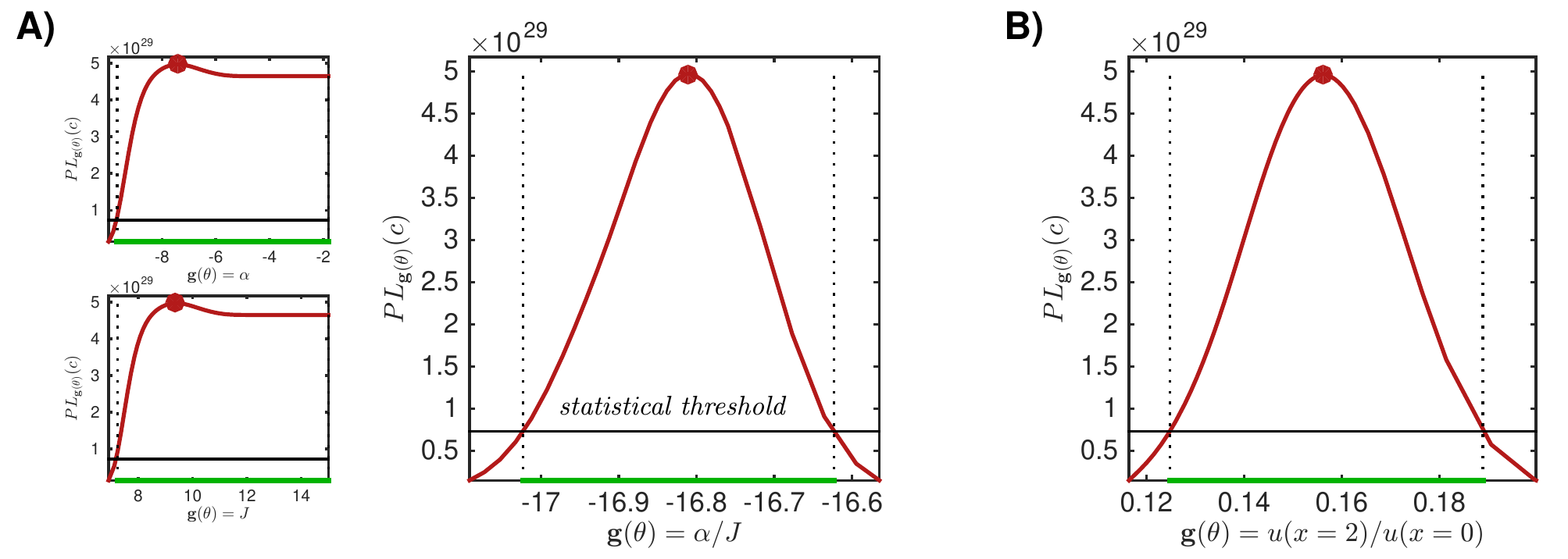}
\caption{\textbf{Profile likelihood calculation for parameter ratio and model prediction.} (A) Profile likelihoods (\textcolor{ProfRed}{\rule[.8mm]{0.25cm}{2pt}}) for $\alpha$ and $\beta$ revealing practical non-identifiability. Profile likelihood for the parameter combination $\mathbf{G}(\theta,u) = \frac{\alpha}{\beta}$ with a finite confidence interval $\mathrm{CI}_{0.95,\mathbf{g}(\theta)}$ (\textcolor{ConfGreen}{\rule[.8mm]{0.25cm}{2pt}}). (B) Profile likelihood and confidence interval for $\mathbf{G}(\theta,u) = u(t=100,x=2)/u(t=100,x=0)$. All profile likelihoods are computed using the ODE formulation of the integration based method with the Hessian.}
\label{fig:App_ProfComb}
\end{figure}

\section{Conclusion}
\label{sec:Conclusion}

In this paper we considered profile likelihood methods for uncertainty analysis in PDE constrained inverse problems. Profile likelihoods provide statistically interpretable confidence bounds for parameters and model predictions using maximum projections of the likelihood. We formulated optimization based profile likelihood calculation methods for the reduced and the full problem. As optimization based approaches can become computationally demanding for PDE constrained problems, we extended the results for the reduced problem by Chen and Jennrich~\cite{ChenJen2002} to PDEs. In addition, we formulated an integration based profile likelihood calculation method for the full problem and established equivalence with the reduced formulation.

The integration based methods provide the exact profile likelihood, if second order information, i.e. Hessian-vector products, are available. We introduced an approximation of the integration based approach for the full problem to circumvent the second order information. A bound for the approximation error was provided and convergence with respect to the retraction factor was established. 

Optimization and integration based profile likelihood calculation methods for the reduced problem were assessed for a semi-linear PDE model of gradient formation in fission yeast. For both, practically identifiable and practically non-identifiable parameters, integration based methods were 4- to 5-fold faster than state-of-the-art optimization based approaches. The precise speed-up depended on the specific implementation, using exact second order information improved accuracy as well as computational efficiency. We note that for the practically non-identifiable parameters of the model, uncertainty analysis methods based on local approximations fail to provide realistic confidence bounds~\cite{MurphyVaa2000}.

The implementation used for the comparison exploited forward methods for the calculation of gradient and Hessian of the objective function. We expect a further improvement of the computational efficiency by using adjoint methods. In addition, the implementation of integration based profile likelihood calculation for the full formulation promises an improved computational efficiency. The simulation of the PDE model -- currently performed in each solver step -- would be circumvented and instead the coupled ODE-PDE system~\req{eq:PL_diff_unred} would be simulated. In addition, the implementation could be extended to the profile likelihood analysis of time-dependent properties (see~\cite{HassKre2016}).

In summary, this study presented optimization and integration based profile likelihood calculation methods for PDE constrained problems. Besides exact methods, we present approximation and corresponding error bounds. The methods for the reduced formulation are  implemented in the open-source MATLAB toolbox PESTO, which will facilitate the application of the method and simplify the development of extensions. This is particularly interesting as the methods we developed can be easily transferred to a broad class of PDE models. Accordingly, this study can help to improve uncertainty analysis in a number of scientific fields.

\section*{Appendix}
Proof of Proposition \ref{prop:retr_full}.\\
Assumptions \req{Wsymm} and \req{Wpos} imply that $(\xi,\zeta)_{W_c}=\dup{W_c \xi}{\zeta}$ defines an inner product on the space $X$.

For any $c\in[c_0,c_1]$, we abbreviate $\tilde{g}_c=W_c^{-1}\nabla_{\xi}\mathbf{G}(\hat{\xi}_c)\in X$ and $\tilde{l}_c=W_c^{-1}\nabla_{\xi}\mathcal{L}(\hat{\xi}_c)\in X$, and define 
the mapping $P_c:X\to X$ by 
\[
P_c\zeta=\frac{\dup{\nabla_{\xi}\mathbf{G}(\hat{\xi}_c)}{\zeta}}{\dup{\nabla_{\xi}\mathbf{G}(\hat{\xi}_c)}{\tilde{g}_c}}\tilde{g}_c
\] 
Note that $P_c$ is linear and idempotent, i.e., a projection onto the one-dimensional linear space $\mbox{span}(\tilde{g}_c)$, actually the orthogonal projection with respect to the inner product $(\cdot,\cdot)_{W_c}$. The second line in \req{eq:PL_diff_unred_retr_x} yields 
\[
P\dot{\hat{\xi}_c}=\frac{1}{\dup{\nabla_{\xi}\mathbf{G}(\hat{\xi}_c)}{\tilde{g}_c}}\tilde{g}_c
\]
and the first line in \req{eq:PL_diff_unred_retr_x} (after application of $W_c^{-1}$) together with the fact that $(I-P_c) \tilde{g}_c=0$ yields
\[
(I-P)\dot{\hat{\xi}_c}=-\gamma(I-P_c)\tilde{l}_c
\]
hence altogether we have eliminated $\hat{\lambda}_c$ and end up with the identity  
\[
\dot{\hat{\xi}}_c=\frac{1}{\dup{\nabla_{\xi}\mathbf{G}(\hat{\xi}_c)}{\tilde{g}_c}}\tilde{g}_c-\gamma(I-P_c)\tilde{l}_c\,.
\]
i.e., the following evolution equation for the error:
\begin{equation}\label{dothatxc}
\dot{\hat{\xi}}_c-\dot{\xi}_c=\frac{1}{\dup{\nabla_{\xi}\mathbf{G}(\hat{\xi}_c)}{\tilde{g}_c}}\tilde{g}_c-\dot{\xi}_c-\gamma(I-P_c)\tilde{l}_c\,.
\end{equation}
Here, the last term is responsible for retraction. Indeed, using symmetry \req{Wsymm} and the definition of $P_c$, which yields, for any $\xi,\zeta\in X$
\[
\dup{W_c\xi}{(I-P_c)\zeta}=\dup{W_c\zeta}{(I-P_c)\xi}
\] 
as well as 
\begin{equation}\label{ImPx}
\dup{W_c\tilde{g}_c}{(I-P_c)\xi}=\dup{W_c \xi}{(I-P_c)\tilde{g}_c}=\dup{W_c \xi}{0}=0
\end{equation}
we get
\[
\eqalign{
\dup{W_c(\hat{\xi}_c-\xi_c)}{(I-P_c)\tilde{l}_c}
= \dup{W_c\tilde{l}_c}{(I-P_c)(\hat{\xi}_c-\xi_c)}\\
= \dup{W_c(\tilde{l}_c+\lambda_c\tilde{g}_c)}{(I-P_c)(\hat{\xi}_c-\xi_c)}\\
= \dup{\nabla_{\xi}\mathcal{L}(\hat{\xi}_c)+\lambda_c\nabla_{\xi}\mathbf{G}(\hat{\xi}_c)}{(I-P_c)(\hat{\xi}_c-\xi_c)}
}
\]
Since $\nabla_{\xi}\mathcal{L}(\xi_c)+\lambda_c\nabla_{\xi}\mathbf{G}(\xi_c)=0$, 
we have 
\[\eqalign{
\dup{\nabla_{\xi}\mathcal{L}(\hat{\xi}_c)+\lambda_c\nabla_{\xi}\mathbf{G}(\hat{\xi}_c)}{(I-P_c)(\hat{\xi}_c-\xi_c)}\\
=
\dup{(\nabla_{\xi}^2\mathcal{L}(\xi_c)+\lambda_c\nabla_{\xi}^2\mathbf{G}(\xi_c))(\hat{\xi}_c-\xi_c)}{(I-P_c)(\hat{\xi}_c-\xi_c)}\\
\quad+ \dup{\mbox{tay}}{(I-P_c)(\hat{\xi}_c-\xi_c)}
}\]
for
\[
\mbox{tay}=\nabla_{\xi}\Phi(\hat{\xi}_c)-\nabla_{\xi}\Phi(\xi_c)
-\nabla_{\xi}^2\Phi(\xi_c)(\hat{\xi}_c-\xi_c)
=o(\|\hat{\xi}_c-\xi_c\|_X)
\]
where $\Phi(\xi)=\mathcal{L}(\xi)+\lambda_c\mathbf{G}(\xi)$.
Moreover,
$(I-P_c)(\hat{\xi}_c-\xi_c)\in \nabla_{\xi} \mathbf{G}(\xi_c)_\bot$ (cf. \req{ImPx}), and by $\dup{\nabla_{\xi}\mathbf{G}(\hat{\xi}_c)}{\hat{\xi}_c-\xi_c}\approx\mathbf{G}(\hat{\xi}_c)-\mathbf{G}(\xi_c)=c-c=0$, we have 
\[
\|P_c(\hat{\xi}_c-\xi_c)\|_X
=\frac{|\dup{\nabla_{\xi}\mathbf{G}(\hat{\xi}_c)}{\hat{\xi}_c-\xi_c}|}{|\dup{\nabla_{\xi}\mathbf{G}(\hat{\xi}_c)}{\tilde{g}_c}|} 
\|\tilde{g}_c\|_X
\leq \bar{M}_c  \, \|\hat{\xi}_c-\xi_c\|_X^2
\]
for $\bar{M}_c=\frac{\sup_{\xi\in[\xi_c,\hat{\xi}_c]} \|\nabla_{\xi}^2\mathbf{G}(\xi)\|}{2\gamma_W\|\tilde{g}_c\|_X}$.
Thus by \req{ssc}we can further estimate
\[
\eqalign{
\dup{W_c(\hat{\xi}_c-\xi_c)}{(I-P_c)\tilde{l}_c}
\geq\gamma_L\|(I-P_c)(\hat{\xi}_c-\xi_c)\|_X^2 -\gamma_L\bar{M}_c^2  \, \|\hat{\xi}_c-\xi_c\|_X^4\\
\quad+ \dup{\mbox{tay}}{(I-P_c)(\hat{\xi}_c-\xi_c)}\\
\quad-\|\nabla_{\xi}^2\mathcal{L}(\xi_c)+\lambda_c\nabla_{\xi}^2\mathbf{G}(\xi_c)\| \, \bar{M}_c  \, \|\hat{\xi}_c-\xi_c\|_X^2 \|(I-P_c)(\hat{\xi}_c-\xi_c)\|_X\\
\geq\frac{\gamma_L}{2}\|\hat{\xi}_c-\xi_c\|_X^2\\
\quad+ \dup{\mbox{tay}}{(I-P_c)(\hat{\xi}_c-\xi_c)}\\
\quad-\|\nabla_{\xi}^2\mathcal{L}(\xi_c)+\lambda_c\nabla_{\xi}^2\mathbf{G}(\xi_c)\| \, \bar{M}_c  \, \|\hat{\xi}_c-\xi_c\|_X^2 \|(I-P_c)(\hat{\xi}_c-\xi_c)\|_X
}
\]
Thus applying $W_c(\hat{\xi}_c-\xi_c)$ to $\dot{\hat{\xi}}_c-\dot{\xi}_c$ and using the identity 
\[
\eqalign{\dup{W_c(\hat{\xi}_c-\xi_c)}{\dot{\hat{\xi}}_c-\dot{\xi}_c}\\
=\frac12 \frac{d}{dc} \dup{W_c(\hat{\xi}_c-\xi_c)}{\hat{\xi}_c-\xi_c} 
-\frac12 \dup{\dot{W}_c(\hat{\xi}_c-\xi_c)}{\hat{\xi}_c-\xi_c} 
}
\]
that follows from symmetry \req{Wsymm},
we get from \req{dothatxc}
\[
\hspace*{-2.5cm}
\eqalign{
\frac12 \frac{d}{dc} \dup{W_c(\hat{\xi}_c-\xi_c)}{\hat{\xi}_c-\xi_c}\\
\leq-\gamma\frac{\gamma_L}{2}\|\hat{\xi}_c-\xi_c\|_X^2 +\gamma\gamma_L\bar{M}_c^2  \, \|\hat{\xi}_c-\xi_c\|_X^4\\
\quad- \gamma\dup{\mbox{tay}}{(I-P_c)(\hat{\xi}_c-\xi_c)}\\
\quad+\gamma\|\nabla_{\xi}^2\mathcal{L}(\xi_c)+\lambda_c\nabla_{\xi}^2\mathbf{G}(\xi_c)\| \, \bar{M}_c  \, \|\hat{\xi}_c-\xi_c\|_X^2 \|(I-P_c)(\hat{\xi}_c-\xi_c)\|_X\\
\quad + \frac12 \|\dot{W}_c\| \, \|\hat{\xi}_c-\xi_c\|_X^2 
+\left(\frac{1}{\|\tilde{g}_c\|_X}+\sqrt{\dup{W_c\dot{\xi}_c}{\dot{\xi}_c}} \right) 
\sqrt{\dup{W_c(\hat{\xi}_c-\xi_c)}{\hat{\xi}_c-\xi_c}}\,,
}
\]
where by Young's inequality, the last term can be bounded by 
\[\frac{\tilde{\epsilon}}{2}
+\frac{1}{2\tilde{\epsilon}} \left(\frac{1}{\|\tilde{g}_c\|_X}+\sqrt{\dup{W_c\dot{\xi}_c}{\dot{\xi}_c}} \right)^2
\dup{W_c(\hat{\xi}_c-\xi_c)}{\hat{\xi}_c-\xi_c}\,,
\]
where $\tilde{\epsilon}>0$ can still be chosen.
Thus using \req{Wpos}, we end up with an estimate of the form
\begin{equation}\label{estedot}
\dot{e}_c\leq \tilde{\epsilon}-\left(\gamma m-\gamma f(e_c)-M_c^{\tilde{\epsilon}} \right) e_c 
\end{equation}
for $e_c=\dup{W_c(\hat{\xi}_c-\xi_c)}{\hat{\xi}_c-\xi_c}$, where 
$m=\frac{\gamma_L}{\bar{M}_W}$, $M_c^{\tilde{\epsilon}}=\frac{\|\dot{W}_c\|}{\gamma_W}+\frac{1}{\tilde{\epsilon}}\left(\frac{1}{\|\tilde{g}_c\|_X}+\sqrt{\dup{W_c\dot{\xi}_c}{\dot{\xi}_c}}\right)^2$, and 
$f(t)=o(t)$ as $t\to0$ and wlog $f$ is monotonically decreasing.\\
So there exists $\rho>0$ such that $f(\rho)<\frac{m}{2}$.
We impose the initial smallness $e_{c_0}<\rho$ and, for any $\tilde{\epsilon}\in]0,\frac{\rho-e_{c_0}}{c_1-c_0}[$, choose $\gamma\geq \frac{2}{m} M_c^{\tilde{\epsilon}}$. With this choice we have, first of all, that $e_c\leq \rho$ for all $c\in[c_0,c_1]$, which can be seen as follows: Assume, on the contrary, that for some $c\in[c_0,c_1]$, $e_c>\rho$ holds and define $c_2$ to be the smallest such $c$, $c_2=\inf\{c\in[c_0,c_1]\, : \ e_c> \rho\}$. Then by the initial smallness condition, $c_2$ must be strictly larger than $c_0$, by minimality of $c_2$ we have $e_c\leq \rho$ for all $c\in[c_0,c_2]$, and finally, by the sequential defintion of the infimum we get $e_{c_2}\geq\rho$. Integration of \req{estedot} therefore by the choice of $\rho$ and $\gamma$ as well as the initial smallness condition yields
\[
\eqalign{
e_{c_2}\leq e_{c_0}+\tilde{\epsilon}(c_2-c_0)-\int_{c_0}^{c_2}\left(\gamma m-\gamma f(e_c)-M_c^{\tilde{\epsilon}} \right) e_c \, dc\\
\qquad \leq 
e_{c_0}+\tilde{\epsilon}(c_2-c_0)-\int_{c_0}^{c_2}\left(\gamma \frac{m}{2}-M_c^{\tilde{\epsilon}} \right) e_c \, dc\\
\qquad \leq 
e_{c_0}+\tilde{\epsilon}(c_2-c_0)<\rho\,,
}
\]
which contradicts $e_{c_2}\geq\rho$. Thus we have shown the boundedness estimate in \req{bdexpdecay}, which additionally implies that $f(e_c)\leq\frac{m}{2}$ for all $c\in[c_0,c_2]$ and hence 
\begin{equation}\label{estedot1}
\dot{e}_c\leq \tilde{\epsilon}-\left(\gamma \frac{m}{2}-M_c^{\tilde{\epsilon}} \right) e_c 
\end{equation}

To prove the exponential decay estimate in  \req{bdexpdecay} for given $\kappa>0$, $\tilde{\epsilon}\in]0,\frac{\rho-e_{c_0}}{c_1-c_0}[$, we choose $\lambda$ possibly larger, namely $\lambda\geq \frac{2}{m} (\kappa+M_c^{\tilde{\epsilon}})$ to obtain from \req{estedot1}
\begin{equation}\label{estedot2}
\dot{e}_c\leq \tilde{\epsilon}-\kappa e_c 
\end{equation}
and Gronwall's inequality, applied to $e_c-\frac{\tilde{\epsilon}}{\kappa}$, that
\[
e_c-\frac{\tilde{\epsilon}}{\kappa}\leq \left(e_{c_0}-\frac{\tilde{\epsilon}}{\kappa}\right)\exp(-\kappa(c-c_0))\,. 
\]

Finally, \req{conv} follows by choosing $\tilde{\epsilon}\leq\frac{\varepsilon}{2}$, $\kappa\geq\max\left\{1, \frac{\ln(2e_{c_0})-\ln(\varepsilon)}{c-c_0}\right\}$, $\lambda\geq \frac{2}{m} (\kappa+M_c^{\tilde{\epsilon}})$.

\end{document}